\documentclass{amsart}

\usepackage{amsmath,amsthm,amssymb,amsfonts}
\usepackage{enumitem}
\usepackage{graphicx,color}
\usepackage[english]{babel}
\usepackage{tikz}
\usepackage{pgfplots}

\usepackage{cleveref}
  \crefname{theorem}{Theorem}{Theorems}
  \crefname{lemma}{Lemma}{Lemmas}
  \crefname{remark}{Remark}{Remarks}
  \crefname{prop}{Proposition}{Propositions}
  \crefname{definition}{Definition}{Definitions}
  \crefname{corollary}{Corollary}{Corollaries}
  \crefname{conjecture}{Conjecture}{Conjectures}
  \crefname{question}{Question}{Questions}
  \crefname{section}{Section}{Sections}
  \crefname{figure}{Figure}{Figures}

\numberwithin{equation}{section}

\newtheorem{theorem}{Theorem}[section]
\newtheorem*{theorem*}{Theorem}
\newtheorem*{question*}{Question}
\newtheorem{lemma}[theorem]{Lemma}
\newtheorem{proposition}[theorem]{Proposition}

\newtheorem{question}[theorem]{Question}

\newtheorem{corollary}[theorem]{Corollary}

\theoremstyle{definition}

\newtheorem{remark}[theorem]{Remark}
\newtheorem{definition}[theorem]{Definition}

\DeclareMathOperator{\diam}{diam}

\DeclareMathOperator{\graph}{graph}

\DeclareMathOperator{\Var}{Var}

\DeclareMathOperator{\supp}{supp}

\newcommand{\N}{\mathbb{N}}

\newcommand{\R}{\mathbb{R}}
\newcommand{\E}{\mathbb{E}}
\renewcommand{\P}{\mathbb{P}}
\newcommand{\Z}{\mathbb{Z}}

\newcommand{\iA}{\mathcal{A}}
\newcommand{\iB}{\mathcal{B}}

\newcommand{\iD}{\mathcal{D}}
\newcommand{\iE}{\mathcal{E}}

\newcommand{\iH}{\mathcal{H}}
\newcommand{\iI}{\mathcal{I}}

\newcommand{\iL}{\mathcal{L}}
\newcommand{\iM}{\mathcal{M}}
\newcommand{\iR}{\mathcal{R}}
\newcommand{\iF}{\mathcal{F}}
\newcommand{\iP}{\mathcal{P}}
\newcommand{\iS}{\mathcal{S}}

\newcommand{\iT}{\mathcal{T}}

\newcommand{\iZ}{\mathcal{Z}}

\newcommand{\eps}{\varepsilon}

\begin{document}

\title{Restrictions of H\"older continuous functions}

\author{Omer Angel}

\address{Department of Mathematics, University of British Columbia, Vancouver, BC V6T 1Z2, Canada}
\email{angel@math.ubc.ca}
\thanks{O.A.\ was supported in part by NSERC. R.B.\ and A.M.\ were supported by the National Research, Development and Innovation Office-NKFIH, 104178. A.M.\ was also supported by the Leverhulme Trust.}

\author{Rich\'ard Balka}

\address{Department of Mathematics, University of British Columbia, and Pacific Institute for the Mathematical Sciences, Vancouver, BC V6T 1Z2, Canada}

\email{balka@math.ubc.ca}

\author{Andr\'as M\'ath\'e}
\address{Mathematics Institute,
University of Warwick, Coventry, CV4 7AL, United Kingdom}
\email{A.Mathe@warwick.ac.uk}

\author{Yuval Peres}
\address{Microsoft Research, 1 Microsoft Way, Redmond, WA 98052, USA}
\email{peres@microsoft.com}



\begin{abstract}
For $0<\alpha<1$ let $V(\alpha)$ denote the supremum of the numbers $v$
such that every $\alpha$-H\"older continuous function is of bounded
variation on a set of Hausdorff dimension $v$. Kahane and Katznelson 
(2009) proved the estimate $1/2 \leq V(\alpha)\leq 1/(2-\alpha)$ and
asked whether the upper bound is sharp. We show that in fact
$V(\alpha)=\max\{1/2,\alpha\}$. Let $\dim_{\iH}$ and $\overline{\dim}_\iM$
denote the Hausdorff and upper Minkowski dimension, respectively. The
upper bound on $V(\alpha)$ is a consequence of the following theorem. Let
$\{B(t): t\in [0,1]\}$ be a fractional Brownian motion of Hurst index
$\alpha$. Then, almost surely, there exists no set $A\subset [0,1]$ such
that $\overline{\dim}_{\iM} A>\max\{1-\alpha,\alpha\}$ and $B\colon A\to \mathbb{R}$
is of bounded variation. Furthermore, almost surely, there exists no set
$A\subset [0,1]$ such that $\overline{\dim}_\iM A>1-\alpha$ and $B\colon
A\to \mathbb{R}$ is $\beta$-H\"older continuous for some $\beta>\alpha$. The
zero set and the set of record times of $B$ witness that the above
theorems give the optimal dimensions. We also prove similar restriction
theorems for deterministic self-affine functions and generic 
$\alpha$-H\"older continuous functions.

Finally, let $\{\mathbf{B}(t): t\in [0,1]\}$ be a two-dimensional Brownian motion. We prove that, almost surely, there is a compact set
$D\subset [0,1]$ such that $\dim_\iH D\geq 1/3$ and $\mathbf{B}\colon D\to \mathbb{R}^2$ is non-decreasing in each coordinate.
It remains open whether $1/3$ is best possible.
\end{abstract}

\maketitle

\section{Introduction}

Let $C^{\alpha}=C^{\alpha}[0,1]$ denote the set of
$\alpha$-H\"older continuous functions $f\colon [0,1]\to \R$.
In 2009, Kahane and Katznelson \cite{KK} proved the following result and asked whether it is sharp.

\begin{theorem*}[Kahane--Katznelson] For every $0<\alpha<1$ there exists a function $g_{\alpha}\in C^{\alpha}$ such that if $A\subset [0,1]$ and $g_{\alpha}|_{A}$ is of bounded variation, then the Hausdorff dimension satisfies $\dim_\iH A\leq 1/(2-\alpha)$.
\end{theorem*}

\begin{question*}[Kahane--Katznelson] Is the above result the best possible?
 \end{question*}

We answer this question negatively and determine the optimal bound. Let
\[
V(\alpha)=\inf_{f\in C^{\alpha}} \sup_{A\subset [0,1]}\{\dim_\iH A:
f|_{A} \textrm{ is of bounded variation}\},
\]
so that the above theorem states $V(\alpha) \leq 1/(2-\alpha)$, see Figure~\ref{f:V}.

\begin{theorem}\label{t:Va}
For all $0<\alpha<1$ we have
\[
  V(\alpha) = \max\left\{1/2, \alpha\right\}.
\]
\end{theorem}

 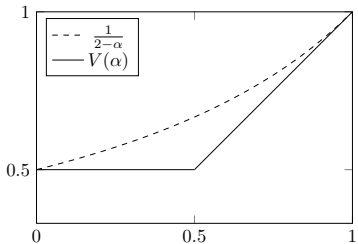
\begin{figure}
	\begin{tikzpicture}[scale=.7]
	\begin{axis}[xtick={0,0.5,1}, ytick={0.5,1},xmin=0, xmax=1,
	ymin=0.33, ymax=1, x=6cm, y=6cm, legend pos=north west]
	\addplot[domain=0:1, dashed]{1/(2-x)};
	\addplot[domain=0:0.5]{0.5};
	\addplot[domain=0.5:1]{x};
	\legend{$\frac{1}{2-\alpha}$,$V(\alpha)$};
	\end{axis}
	\draw (6.7,0);
	\end{tikzpicture}
 \caption{The Kahane--Katznelson bound $V(\alpha)\leq 1/(2-\alpha)$ compared to the actual value $V(\alpha)=\max\{1/2,\alpha\}$.}
 \label{f:V}
 \end{figure}

\begin{figure}
\centering
\includegraphics[width=0.9\textwidth]{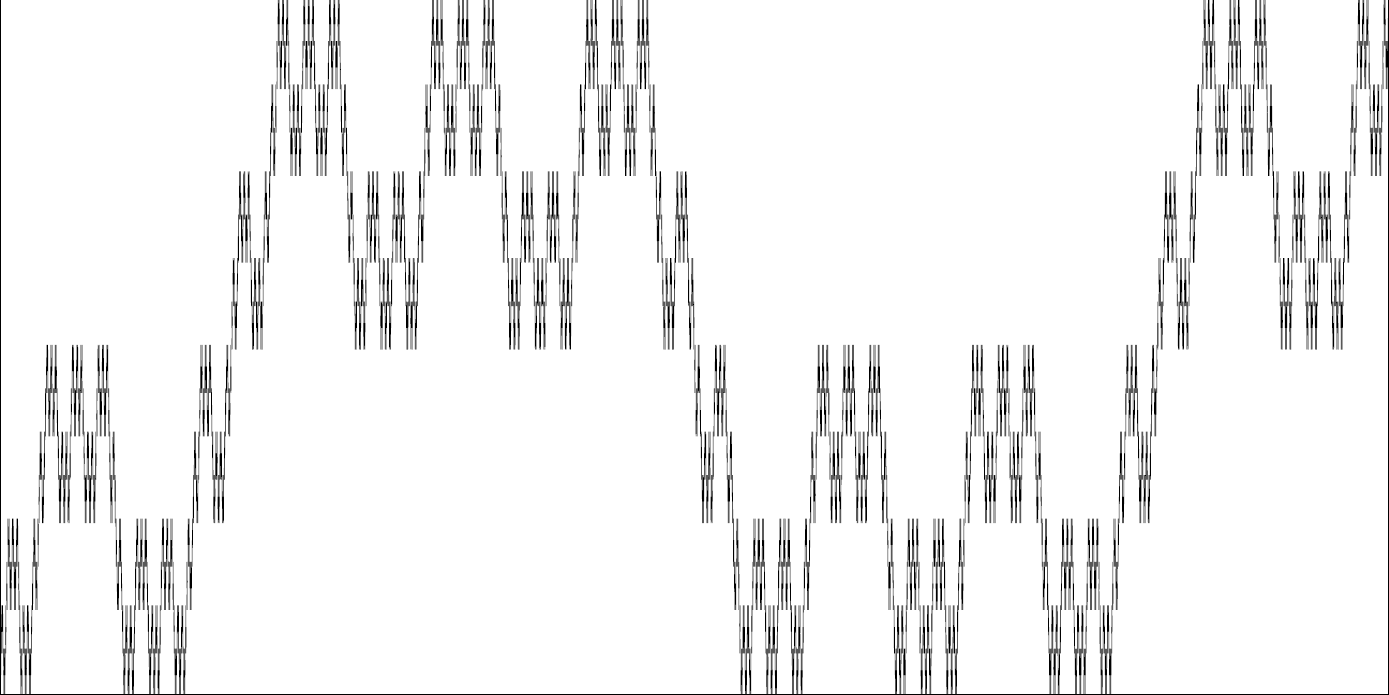}
\caption{A self-affine function $f_\alpha\colon [0,1]\to [0,1]$ with H\"older exponent
\mbox{$\alpha=\log 2/\log 6$}. Its graph consists of 6 affine copies of itself.}
\label{f:23}
\end{figure}

Kahane and Katznelson also asked about dimensions of sets $A$ such that
the restriction to $A$ is H\"older continuous. (See the next section
for related results.) We present two constructions, one deterministic
and one stochastic, of functions that are not H\"older on any set of
high enough dimension. First we consider self-affine functions. These
are constructed in Definition~\ref{d:sa} below, see Figure~\ref{f:23} for
illustration.

\begin{theorem} \label{t:sa0} There is a dense set $\Delta \subset (0,1)$ with the following property. For
each $\alpha\in \Delta$ there is a self-affine function $f_\alpha \in C^{\alpha}$
such that for all $A\subset [0,1]$
\begin{enumerate}
\item if $f_{\alpha}|_A$ is $\beta$-H\"older continuous for some $\beta>\alpha$, then $\overline{\dim}_\iM A\leq 1-\alpha$;
\item if $f_{\alpha}|_A$ is of bounded variation, then $\overline{\dim}_\iM A\leq \max\{1-\alpha,\alpha\}$.
\end{enumerate}
\end{theorem}

For a stochastically self-affine process, fractional Brownian motion
(see Definition~\ref{d:fBm}), we prove the following.

\begin{theorem} \label{t:fBm0} Let $0<\alpha<1$ and let $\{B(t): t\in [0,1]\}$ be a fractional Brownian motion of Hurst index $\alpha$.
Then, almost surely, for all $A\subset [0,1]$
\begin{enumerate}
\item if $B|_{A}$ is $\beta$-H\"older continuous for some $\beta>\alpha$, then $\overline{\dim}_\iM A\leq 1-\alpha$;
\item if $B|_{A}$ is of bounded variation, then $\overline{\dim}_\iM A\leq \max\{1-\alpha, \alpha\}$.
\end{enumerate}
\end{theorem}

\begin{corollary} \label{c:fBm}
Let $0<\alpha<1$ and let $\{B(t): t\in [0,1]\}$ be a
fractional Brownian motion of Hurst index $\alpha$. Then
\[\mathbb{P}(\exists A: \overline{\dim}_\iM A>\max\{1-\alpha,\alpha\} \textrm{ and } B|_{A} \textrm{ is non-decreasing})=0.\]
\end{corollary}

Let $\iZ$ be the zero set of $B$ and let 
$\iR=\{t\in [0,1]: B(t)=\max_{s\in [0,t]} B(s)\}$ be the set of record times of $B$. It is classical that,
almost surely, $\dim_\iH \iZ=1-\alpha$, see \cite[Chapter~18]{Ka}. For the
record, let us state the following, more subtle fact.

\begin{proposition}\label{p:R}
Almost surely, $\dim_\iH \iR =\overline{\dim}_\iM \iR =\alpha$.
\end{proposition}

We could not find a reference for this in the literature, and include a
proof in Section~\ref{s:fBm}. Clearly $\iZ$ and $\iR$ witness that Theorem~\ref{t:fBm0} and
Corollary~\ref{c:fBm} are best possible.

Simon~\cite{S} proved that a standard linear Brownian motion is not monotone on any set of positive Lebesgue measure.
Theorem~\ref{t:fBm0} for $\alpha=1/2$ with Hausdorff dimension in place of upper Minkowski dimension is due to Balka and Peres~\cite{BP}.
The methods used there do not extend to Minkowski dimension or to general exponents $\alpha$.
Related results in the discrete setting, concerning non-decreasing subsequences of random walks, can be found in \cite{ABP}.

Now we consider higher dimensional Brownian motion.

\begin{definition} Let $d\geq 2$ and $f\colon [0,1]\to \R^d$. We say that $f$ is \emph{non-decreasing}
on a set $A\subset [0,1]$ if all the coordinate functions of $f|_{A}$ are non-decreasing.
\end{definition}

\begin{theorem} \label{t:2DBM} Let $\{B(t): t\in [0,1]\}$ be a standard two-dimensional Brownian motion. Then, almost surely,
there exists a compact set $D\subset [0,1]$ such that $B$ is non-decreasing on $D$ and $\dim_\iH D\geq 1/3$.
\end{theorem}

Corollary~\ref{c:fBm} (or \cite[Theorem~1.2]{BP}) implies that, almost surely, the $d$-dimensional Brownian motion $B$
cannot be non-decreasing on any set of Hausdorff dimension larger than $1/2$.
The following problem remains open in all dimensions $d\geq 2$.

\begin{question} Let $d\geq 2$ and let $\{B(t)\colon 0\leq t\leq 1\}$ be a standard $d$-dimensional Brownian motion. What is the supremum of
the numbers $\gamma$ such that, almost surely, $B$ is non-decreasing on some set of Hausdorff dimension $\gamma$?
\end{question}

Finally, we prove restriction theorems for a generic $\alpha$-H\"older continuous function (in the sense of Baire category), see the following section for the details.

\section{Related work and general statements} \label{s:rel}

Let $C[0,1]$ denote the set of continuous functions $f\colon [0,1]\to \R$ endowed with the maximum norm.
Elekes~\cite[Theorems~1.4,~1.5]{E} proved the following restriction theorem.

\begin{theorem}[Elekes]
Assume that $0<\beta<1$. For a generic continuous function $f\in C[0,1]$ (in the sense of Baire category) for all $A\subset [0,1]$
\begin{enumerate}
\item if $f|_{A}$ is $\beta$-H\"older continuous, then $\dim_\iH A\leq 1-\beta$;
\item if $f|_{A}$ is of bounded variation, then $\dim_\iH A\leq 1/2$.
\end{enumerate}
\end{theorem}

Kahane and Katznelson~\cite[Theorems~2.1,~3.1]{KK},
and independently M\'ath\'e~\cite[Theorems~1.4,~1.5]{M} proved that
the above result is sharp.

\begin{theorem}[Kahane--Katznelson, M\'ath\'e]
\label{t:KKM}
Let $0<\beta<1$. For every $f\in C[0,1]$ there are compact sets $A,C\subset [0,1]$ such that
  \begin{enumerate}
  \item $\dim_\iH A=1-\beta$ and $f|_{A}$ is $\beta$-H\"older continuous;
  \item $\dim_\iH C=1/2$ and $f|_{C}$ is of bounded variation.
  \end{enumerate}
\end{theorem}

In other words there is always a set $A$ with the given properties and
dimension, and for generic functions there is no $A$ of larger dimension.
Let us recall that the \emph{$\beta$-variation} of a function
$f\colon A\to \R$ is defined as
\[V^{\beta} (f)=\sup\left\{\sum_{i=1}^{n} |f(x_{i})-f(x_{i-1})|^{\beta}: x_0<\dots <x_n,~x_i\in A,~n\in \N^+\right\}.\]
In the theorems above, bounded variation can be generalized to finite
$\beta$-variation for all $\beta>0$ by similar methods. For the
following result see M\'ath\'e~\cite[Theorem~5.2]{M}.

\begin{theorem}[M\'ath\'e]\label{t:M}
Let $\beta>0$ and $f\in C[0,1]$. Then there is a compact set $A\subset [0,1]$ such that
$\dim_\iH A=\beta/(\beta+1)$ and $f|_{A}$ has finite $\beta$-variation.
\end{theorem}

Our initial interest came from questions of Kahane and Katznelson~\cite{KK}
on restrictions of H\"older continuous functions. First we need the following definition.

\begin{definition}
Let $C^{\alpha}(A)$ be the set of $\alpha$-H\"older continuous functions $f\colon A\to \R$.
For all $0<\alpha<1$ and $\beta>0$ define
  \begin{align*}
    H(\alpha,\beta)&=\inf_{f\in C^{\alpha}[0,1]} \sup_{A\subset [0,1]} \left\{\dim_\iH A: f|_{A}\in C^{\beta}(A)\right\}, \\
    V(\alpha,\beta)&=\inf_{f\in C^{\alpha}[0,1]} \sup_{A\subset [0,1]} \left\{\dim_\iH A: V^{\beta}(f|_{A})<\infty\right\}.
  \end{align*}
  Replacing Hausdorff dimension by upper Minkowski dimension yields
  \begin{align*}
  \overline{H}(\alpha,\beta)&=\inf_{f\in C^{\alpha}[0,1]} \sup_{A\subset [0,1]} \left\{\overline{\dim}_\iM A: f|_{A}\in C^{\beta}(A)\right\}, \\
  \overline{V}(\alpha,\beta)&=\inf_{f\in C^{\alpha}[0,1]} \sup_{A\subset [0,1]} \left\{\overline{\dim}_\iM A: V^{\beta}(f|_{A})<\infty\right\}.
\end{align*}
\end{definition}

\begin{remark} As the Hausdorff dimension is smaller than or equal to the upper Minkowski dimension,
$H(\alpha,\beta)\leq \overline{H}(\alpha,\beta)$ and $V(\alpha,\beta)\leq \overline{V}(\alpha,\beta)$. If $\beta\geq 1/\alpha$ then $V^{\beta}(f)<\infty$ for all $f\in C^{\alpha}[0,1]$, so $V(\alpha,\beta)=\overline{V}(\alpha,\beta)=1$.
\end{remark}

For the following theorem see \cite[Theorems~5.1,~5.2]{KK}.

\begin{theorem}[Kahane--Katznelson]
  For all $0<\alpha<\beta<1$ we have
  \[H(\alpha,\beta)\leq \frac{1-\beta}{1-\alpha} \quad \textrm{and} \quad
  V(\alpha,1)\leq \frac{1}{2-\alpha}.\]
\end{theorem}

\begin{question}[Kahane--Katznelson] \label{q:KK} Are the above bounds optimal?
\end{question}

We answer this question negatively and find the sharp bounds, which generalizes Theorem~\ref{t:Va}.

\begin{theorem}\label{t:answer}
For all $0<\alpha<1$ we have
  \begin{align*}
    H(\alpha,\beta)&=\overline{H}(\alpha,\beta)=1-\beta \textrm{ for all } \alpha<\beta\leq 1, \\
    V(\alpha,\beta)&=\overline{V}(\alpha,\beta)=\max\left\{\alpha\beta, \beta/(\beta+1)\right\} \textrm{ for all }
    0<\beta<1/\alpha.
  \end{align*}
\end{theorem}

In \cref{s:slte} we prove restriction theorems for functions which satisfy certain
scaled local time estimates. This allows us to prove the following more general version of Theorems~\ref{t:sa0} and \ref{t:fBm0},
see \cref{s:KK,s:fBm}, respectively.

\begin{theorem} \label{t:sa} There is a dense set $\Delta \subset (0,1)$ with the following property. For each $\alpha\in
\Delta$ there is a self-affine function $f_\alpha \in C^{\alpha}[0,1]$ such that for all $A\subset [0,1]$
\begin{enumerate}
\item \label{eq:sa1} if $f_{\alpha}|_A\in C^{\beta}(A)$ for some $\beta>\alpha$, then $\overline{\dim}_\iM A\leq 1-\alpha$;
\item \label{eq:sa2} if $V^{\beta}(f_{\alpha}|_A)<\infty$ for some $\beta>0$, then $\overline{\dim}_\iM A\leq \max\{1-\alpha,\alpha\beta\}$.
\end{enumerate}
\end{theorem}

\begin{theorem} \label{t:fBm} Let $0<\alpha<1$ and let $\{B(t): t\in [0,1]\}$ be a fractional Brownian motion of Hurst index $\alpha$.
Then, almost surely, for all $A\subset [0,1]$
\begin{enumerate}
\item if $B|_{A}\in C^{\beta}(A)$ for some $\beta>\alpha$, then $\overline{\dim}_\iM A\leq 1-\alpha$;
\item \label{eq:fBm2} if $V^{\beta}(B|_{A})<\infty$ for some $\beta>0$, then $\overline{\dim}_\iM A\leq \max\{1-\alpha, \alpha \beta\}$.
\end{enumerate}
\end{theorem}

The zero set of $B$ and the following result (see \cite[Theorem~4.3]{BP}) with Lemma~\ref{l:ub}
witness that Theorem~\ref{t:fBm}~\eqref{eq:fBm2} is sharp for all $\beta\leq 1/\alpha$.

\begin{theorem}\label{t:A}
Let $0<\alpha<1$ and $0<\beta \leq 1/\alpha$ be fixed. Then there is a compact set $A\subset [0,1]$ (which depends only on $\alpha$ and $\beta$) such that $\dim_\iH A=\alpha \beta$ and if $f\colon [0,1]\to \R$ is a function and $c\in \R^+$ such that for all $x,y\in [0,1]$ we have
\begin{equation*} \label{eq:mod}
|f(x)-f(y)|\leq c|x-y|^{\alpha}\log (1/|x-y|),
\end{equation*}
  then $f|_{A}$ has finite $\beta$-variation.
\end{theorem}

In Section~\ref{s:KK} we prove Theorem~\ref{t:answer} by using \cref{t:sa} to obtain the sharp upper bounds for $\overline{H}(\alpha,\beta)$ and $\overline{V}(\alpha,\beta)$. Theorem~\ref{t:fBm} may be used there instead of Theorem~\ref{t:sa}.
Finally, Theorems \ref{t:KKM},~\ref{t:M} and~\ref{t:A} provide the optimal lower bounds for $H(\alpha,\beta)$ and $V(\alpha,\beta)$.

In Section~\ref{s:2D} we consider higher dimensional Brownian motion and prove Theorem~\ref{t:2DBM}.
In order to do so, we establish a general limit theorem for random sequences with i.i.d.\ increments,
which is of independent interest.

Finally, in Section~\ref{s:generic} we study generic $\alpha$-H\"older continuous
functions in the sense of Baire category.

\begin{definition} For $0<\alpha<1$ let $C_1^{\alpha}[0,1]$ be the set of functions $f\colon [0,1]\to \R$ such that
for all $x,y\in [0,1]$ we have
\[|f(x)-f(y)|\leq |x-y|^{\alpha}.\]
Let us endow $C_1^{\alpha}[0,1]$ with the maximum metric, then it is a complete metric
space and hence we can use Baire category arguments.
\end{definition}

We show that a generic $f\in C_1^{\alpha}[0,1]$ witnesses $H(\alpha,\beta)=1-\beta$ and
$V(\alpha,\beta)=\max\{\alpha\beta, \beta/(\beta+1)\}$ for all $\beta$ simultaneously.

\begin{theorem} \label{t:gen} Let $0<\alpha<1$. For a generic $f\in C_1^{\alpha}[0,1]$ for all $A\subset [0,1]$
\begin{enumerate}
\item \label{eq:gen1}
if $f|_{A}\in C^{\beta}(A)$ for some $\alpha<\beta\leq 1$, then
$\dim_\iH A\leq 1-\beta$;
\item \label{eq:gen2} if $V^{\beta}(f|_{A})<\infty$ for some $\beta>0$, then
$\dim_\iH A\leq \max\{\alpha\beta, \beta/(\beta+1)\}$.
\end{enumerate}
\end{theorem}

\section{Preliminaries}
\label{S:background}

Let $A\subset [0,1]$ be non-empty and $\alpha>0$. A function $f\colon A\to\R$ is called (uniformly)
\emph{$\alpha$-H\"older continuous} if there
exists a constant $c \in (0,\infty)$ such that $|f(x)-f(y)|\leq
c|x-y|^{\alpha}$ for all $x,y\in A$.
For the definitions of $C[0,1]$, $C^{\alpha}[0,1]$, and $C^{\alpha}_1[0,1]$ see Section~\ref{s:rel}.
The diameter of $A$ is denoted by $\diam A$. For all $s\geq 0$ the \emph{$s$-Hausdorff content} of $A$ is
\begin{equation*}
\mathcal{H}^{s}_{\infty}(A)=\inf \left\{ \sum_{i=1}^\infty (\diam A_{i})^{s}: A \subset \bigcup_{i=1}^{\infty} A_{i}\right\}.
\end{equation*}
The \emph{Hausdorff dimension} of $A$ is defined as
\[
\dim_\iH A = \inf\{s \ge 0: \mathcal{H}^{s}_{\infty}(A)=0\}.
\]
Let $|F|$ denote the cardinality of the set $F$. Let $M\geq 2$ be an integer. For all $n\in \N$ and $A\subset [0,1]$ define
\begin{align} \label{eq:Nn}
  \iD_n(M)&=\left\{[pM^{-n},(p+1)M^{-n}): p\in \{0,\dots,M^n-1\}\right\}, \notag\\
  \iD_n(A,M)&=\{I\in \iD_n(M): I\cap A\neq \emptyset\}, \\
  N_n(A,M)&=|D_n(A,M)|. \notag
\end{align}
The \emph{upper Minkowski dimension} of $A$ is defined as
\[
\overline{\dim}_\iM A=\limsup_{n\to \infty} \frac{\log N_n(A,M)}{n \log M}.
\]
It is easy to show that this definition is independent of the choice of $M$
and we have $\dim_\iH A\leq \overline{\dim}_\iM A$ for all $A\subset [0,1]$. For more on
these concepts see \cite{F1}.

\begin{definition} \label{d:sa}
A compact set $K\subset \R^2$ is called \emph{self-affine} if for some
$M\geq 2$ there are injective and contractive affine maps $F_1,\dots,F_M
\colon \R^2\to \R^2$ such that
\[
K = \bigcup_{i=1}^{M} F_i(K).
\]
A continuous function $f\in C[0,1]$ is \emph{self-affine} if
$\graph(f)\subset \R^2$ is a self-affine set.
\end{definition}

\begin{definition} \label{d:fBm} Let $0<\alpha<1$. The process $\{B(t):t\geq 0\}$ is called a
\emph{fractional Brownian motion of Hurst index $\alpha$} if
\begin{itemize}
\item $B$ is a Gaussian process with stationary increments;
\item $B(0)=0$ and $t^{-\alpha}B(t)$ has standard normal distribution for every $t>0$;
\item almost surely, the function $t\mapsto B(t)$ is continuous.
\end{itemize}
\end{definition}
The covariance function of $B$ is $\E(B(t)B(s))=(1/2)(|t|^{2\alpha}+|s|^{2\alpha}-|t-s|^{2\alpha})$.
It is well known that almost surely $B$ is $\gamma$-H\"older continuous for
all $\gamma<\alpha$, see \cref{l:ub} below. For more information see \cite[Chapter~8]{A} and \cite[Chapter~18]{Ka}.

Let $X$ be a \emph{complete} metric space. A set is \emph{somewhere dense} if
it is dense in a non-empty open set, otherwise it is \emph{nowhere dense}. We say that $A \subset X$ is
\emph{meager} if it is a countable union of nowhere dense sets, and
a set is called \emph{co-meager} if its complement is meager. By Baire's category theorem a set is co-meager iff it contains a dense
$G_\delta$ set. We say that the \emph{generic} element $x \in X$ has
property $\mathcal{P}$ if $\{x \in X : x \textrm{ has property }
\mathcal{P} \}$ is co-meager.

Let $(\mathcal{K}[0,1],d_{\iH})$ be the set of non-empty compact subsets of
$[0,1]$ endowed with the \emph{Hausdorff metric}, that is, for each $K_1,K_2\in \mathcal{K}[0,1]$ we have
\[d_{H}(K_1,K_2)=\min \left\{r: K_1\subset B(K_2,r) \textrm{ and } K_2\subset B(K_1,r)\right\},\]
where $B(A,r)=\{x\in \R: \exists y\in A \textrm{ such that } |x-y|\leq r\}$. Then $(\mathcal{K}[0,1],d_{H})$
is a compact metric space, see \cite{K} for more on this concept.

Let $\supp(\mu)$ stand for the support of the measure $\mu$.
For $x\in \R$ let $\lfloor x \rfloor$ and $\lceil x \rceil$ denote the lower and upper
integer part of $x$, respectively.

\section{Functions satisfying a scaled local time estimate} \label{s:slte}

In this section we prove restriction theorems for functions satisfying a
scaled local time estimate. First we need some notation.

\begin{definition}
Let $\alpha\in [0,1]$ and an integer $M\geq 2$ be fixed.
Let $n\in \N$ and $0\leq p\leq M^n-1$. A \emph{time interval of order $n$} is of the form
\[
I_{n,p} = [p M^{-n}, (p+1)M^{-n}).
\]
Let $q\in \Z$. A \emph{value interval of order $n$} is of the form
\[
J_{n,q} = [q M^{-\alpha n}, (q+1) M^{-\alpha n}).
\]
For all $0\leq m\leq n$ define
\[
\iI_{n,m,p}=\{I\in \iD_n(M): I\subset I_{m,p}\},
\]
where $\iD_n(M)$ is the set of time intervals of order $n$. Clearly, $|\iI_{n,m,p}|=M^{n-m}$.
\end{definition}

\begin{definition} \label{d:An}
For a function $f\colon [0,1]\to \R$ the \emph{scaled local time} $A_{n,m,p,q}(f)$ is the number of order $n$
intervals in $\iI_{n,m,p}$ in which $f$ takes at least one value in $J_{n,q}$:
\begin{equation*}
A_{n,m,p,q}(f)= \left|\{I\in\iI_{n,m,p} : \exists x\in I, ~ f(x)\in J_{n,q}\}\right|.
\end{equation*}
It is easy to see that if $f$ is $\alpha$-H\"older continuous then for every
$n,m,p$, for some $q$ we have $A_{n,m,p,q}(f) \geq c M^{(1-\alpha)(n-m)}$,
since the function cannot visit too many different value intervals in any
given time interval. Finally, for each $n\in \N^+$ define
\begin{equation*}
\iA_n(\alpha, M)=\{f: A_{n,m,p,q}(f)\leq n^2 M^{(1-\alpha)(n-m)} \textrm{ for all } m\leq n ,\, p< M^{n},\, q\in \Z \}.
\end{equation*}
\end{definition}
Thus the set $\iA_n(\alpha,M)$ includes $\alpha$-H\"older functions with scaled local
times which are not much larger than the minimal values possible given
their continuity. We shall see below that the self-affine functions we
define, as well as fractional Brownian motion belong (almost surely) to
this class.

The main goal of this section is to prove Theorems~\ref{t:Holder} and \ref{t:variation}.

\subsection{H\"older restrictions}

For the notation $\iA_n(\alpha,M)$ and $N_n(A,M)$ see Definition~\ref{d:An} and \eqref{eq:Nn}, respectively.

\begin{theorem} \label{t:Holder}
Let $M\geq 2$ be an integer and let $\alpha\in [0,1]$. Let $f\colon
[0,1]\to \R$ be a function such that $f\in \iA_n(\alpha,M)$ for all large enough $n$.
Assume that $\beta>\alpha$ and $A\subset [0,1]$ such that $f|_{A}$ is $\beta$-H\"older continuous. Then $\overline{\dim}_\iM A\leq 1-\alpha$.
Moreover,
\begin{equation} \label{eq:NAn1}
N_n(A,M)\leq M^{(1-\alpha)n+O(\log^2 n)}.
\end{equation}
\end{theorem}

\begin{proof}
Assume that $A\subset [0,1]$ and $\alpha<\gamma<\beta$ are fixed such that
$f|_A$ is $\beta$-H\"older continuous. Choose $N\in \N^+$ such that $f\in \iA_n(\alpha,M)$
for all $n\geq N$. Clearly it is enough to prove \eqref{eq:NAn1}. By decomposing $A$ into
finitely many pieces of small enough diameters, we may assume that
$f|_{A}$ is $\gamma$-H\"older continuous with H\"older constant $1$, that is,
for all $x,y\in A$ we have
\begin{equation} \label{eq:Ho} |f(x)-f(y)|\leq |x-y|^{\gamma}. \end{equation}
For all $n\in \N$ let
\[
d_n=N_n(A,M).
\]
Let $c=\gamma/\alpha>1$. Assume that $s,t\in \N$ such that $s\leq t \leq \lfloor cs \rfloor$ and $t\geq N$. Now we will prove that
\begin{equation} \label{eq:di}
d_{t}\leq 2d_s t^2 M^{(1-\alpha)(t-s)}.
\end{equation}
Let us fix an arbitrary $I_{s,p}\in \iD_{s}(A,M)$ for some $p$. As
$|\iD_{s}(A,M)|=d_s$, in order to show \eqref{eq:di} it is enough to
prove that
\begin{equation} \label{eq:Iprime}
|\{I\in \iD_t(A,M): I\subset I_{s,p}\}|\leq 2 t^2 M^{(1-\alpha)(t-s)}.
\end{equation}
Inequality~\eqref{eq:Ho} yields $\diam f(I_{s,p}\cap A)\leq M^{-\gamma s}\leq M^{-\alpha t}$, therefore $f(I_{s,p}\cap A)\subset J_{t,q}\cup
J_{t,q+1}$ for some $q\in \Z$. As $f\in \iA_t(\alpha,M)$, we have $A_{t,s,p,q+j}(f)\leq t^2 M^{(1-\alpha)(t-s)}$ for $j\in \{0,1\}$, which
yields \eqref{eq:Iprime}. Hence \eqref{eq:di} follows.

Fix an integer $m_0\geq \max\{N,c/(c-1)\}$ and let $n$ be an arbitrary integer such that
$n>m_0$. For all $i\in \N^+$ let $m_{i}=\min\{n, \lfloor cm_{i-1} \rfloor\}$.
Let $k$ be the minimal number such that $m_{k+1}=n$. Note
that $c\ell >\ell+1$ for every $\ell \geq m_0$, thus such a $k$ exists.
Then the recursion and $m_0\geq c/(c-1)$ yield that
\[
n\geq m_{k} \geq c^{k}m_0-\sum_{i=0}^{k-1} c^i=c^k(m_0-1/(c-1))\geq c^k,
\]
therefore $k\leq \log n/\log c$. Applying \eqref{eq:di} repeatedly and
using that $d_{m_0}\leq M^{m_0}$ and $m_{i+1}\leq cm_i$ we obtain that
\begin{align*}
d_n&\leq d_{m_0}\prod_{i=1}^{k+1} 2 m^2_i M^{(1-\alpha)(m_i-m_{i-1})} \\
&\leq M^{(1-\alpha)n} M^{m_0+k+1} m^2_0 c^{2(1+\dots +(k+1))} \\
&\leq M^{(1-\alpha)n+O(k^2)} \leq M^{(1-\alpha)n+O(\log^2 n)}.
\end{align*}
Hence \eqref{eq:NAn1} follows, and the proof is complete.
\end{proof}

\subsection{Restrictions of finite $\beta$-variation}

The notation $\iA_n(\alpha,M)$ and $N_n(A,M)$ are given in Definition~\ref{d:An} and \eqref{eq:Nn}, respectively.

\begin{theorem} \label{t:variation}
Let $M\geq 2$ be an integer, $\alpha\in [0,1]$ and $\beta>0$. Let
$f\colon [0,1]\to \R$ be a function such that $f\in \iA_n(\alpha,M)$ for all large enough $n$.
Assume that $A\subset[0,1]$ is such that
$f|_{A}$ has finite $\beta$-variation.
Then $\overline{\dim}_\iM A \leq \max\{1-\alpha,\alpha \beta\}=:\gamma$. Moreover,
\begin{equation} \label{eq:NAn2}
N_n(A,M)\leq M^{\gamma n+O(\sqrt{n\log n})}.
\end{equation}
\end{theorem}

\begin{proof}
If the theorem holds for $\beta =(1-\alpha)/\alpha$, then
it holds for every $\beta<(1-\alpha)/\alpha$. Thus we may
assume that $\beta\geq (1-\alpha)/\alpha$, so
$\gamma = \max\{1-\alpha, \alpha\beta\} = \alpha \beta$. Suppose that $A\subset
[0,1]$ such that $f|_A$ has finite $\beta$-variation. Choose $N\in \N^+$
such that $f\in \iA_n(\alpha,M)$ for all $n\geq N$. Clearly it is
enough to prove \eqref{eq:NAn2}. By decomposing $A$ into finitely many
pieces of small enough diameters, we may assume that the $\beta$-variation of $f|_A$
is at most $1$, that is,
\begin{equation} \label{eq:Var}
V^{\beta}(f|_A)\leq 1.
\end{equation}
Let $s,t\in \N$ such that $s<t$ and $t\geq N$. Assume that
$I=I_{p,s}\in \iD_s(A,M)$ and $I$ contains $r$ sub-intervals in $\iD_t(A,M)$. First we prove that
\begin{equation} \label{eq:VB}
V^{\beta}(f|_{A\cap I})\geq \left(\frac{r}{2t^2 M^{(1-\alpha)(t-s)}}-1\right) M^{-\alpha \beta t}.
\end{equation}

Let $\{Q_1,Q_2,\dots,Q_m\}$ be the sub-intervals of $I$ in $\iD_t(A,M)$
such that for every $1\leq i\leq m$ there is an \emph{even} $q_i\in \Z$
such that $f(Q_i\cap A)\cap J_{t,q_i}\neq \emptyset$. We may assume that
$m\geq r/2$, otherwise we switch to odd integers $q_i$ and repeat the
same proof. For all $i\in \{1,\dots,m\}$ choose an $x_i\in Q_i\cap A$ such that
$f(x_i)\in J_{t,q_i}$. We may assume that $x_i<x_j$ whenever $i<j$. Let
$j_1=1$, and if $j_\ell \in \{1,\dots, m\}$ is defined then let
$j_{\ell+1}=\min\{u>j_\ell: q_u\neq q_{j_\ell}\}$ if the minimum exists.
As $f\in \iA_t(\alpha,M)$, for all $i,j\in \{1,\dots,m\}$ we have
\begin{equation*}
|\{i: f(x_i)\in J_{t,q_{j}}\}|\leq A_{t,s,p,q_{j}}(f)\leq t^2 M^{(1-\alpha)(t-s)},
\end{equation*}
so if $j_\ell \leq m-t^2 M^{(1-\alpha)(t-s)}$ then $j_{\ell+1}$ is
defined and $j_{\ell+1}\leq j_{\ell}+ t^2 M^{(1-\alpha)(t-s)}$. Thus the
length of the defined sequence $j_1<\dots <j_k$ satisfies $k\geq r(2t^2
M^{(1-\alpha)(t-s)})^{-1}$. By construction
$|f(x_{j_{\ell+1}})-f(x_{j_\ell})| \geq M^{-\alpha t}$ for all $\ell<k$, which implies \eqref{eq:VB}.

Index the elements $\iD_s(A,M)=\{I_1,\dots, I_{N_s(A,M)}\}$, and
assume that each $I_i$ contains $r_i$ intervals of $\iD_t(A,M)$, so
$\sum_{i=1}^{N_s(A,M)} r_i=N_t(A,M)$. Inequality \eqref{eq:VB} yields that
\begin{align*}
1&\geq V^{\beta}(f|_{A})\geq \sum_{i=1}^{N_s(A,M)} V^{\beta}(f|_{A\cap I_i})  \\
&\geq \sum_{i=1}^{N_s(A,M)} \left(\frac{r_i}{2t^2 M^{(1-\alpha)(t-s)}}-1\right) M^{-\alpha \beta t} \\
&=\left(\frac{N_t(A,M)}{2t^2 M^{(1-\alpha)(t-s)}}-N_s(A,M)\right)M^{-\alpha \beta t}.
\end{align*}
Therefore
\begin{equation} \label{eq:iD}
\frac{N_t(A,M)}{(2t)^2 M^{(1-\alpha)t}}-\frac{N_s(A,M)}{M^{(1-\alpha)s}}\leq \frac{M^{\alpha \beta t}}{M^{(1-\alpha)s}}.
\end{equation}
Now assume that $m,k\in \N^+$ are fixed such that and $m\geq N$, we prove
that
\begin{equation} \label{eq:Dkm}
N_{km}(A,M)\leq M^{\alpha \beta km}(2km)^{2k}\left(1+kM^{(1-\alpha)m}\right).
\end{equation}
For all $0\leq i\leq k$ let
\[
d_i=\frac{N_{im}(A,M)}{(2km)^{2i} M^{(1-\alpha)im}}.
\]
 Applying Inequality~\eqref{eq:iD} for $t=im$ and $s=(i-1)m$, and using
 that $t\leq km$ and $2km\geq 1$ imply that for every $1\leq i\leq k$ we
 have
\begin{equation*}
 d_i-d_{i-1}\leq \frac{M^{\alpha \beta im}}
{M^{(1-\alpha)(i-1)m}(2km)^{2i-2}}\leq M^{(1-\alpha)m} M^{(\alpha \beta-(1-\alpha))im}.
\end{equation*}
As $\alpha \beta\geq 1-\alpha$, the above inequality implies that
\begin{equation} \label{eq:sum}
d_k-d_0=\sum_{i=1}^k (d_i-d_{i-1}) \leq kM^{(1-\alpha)m} M^{(\alpha \beta -(1-\alpha))km}.
\end{equation}
Then $d_0=1$, $\alpha\beta \geq 1-\alpha$ and \eqref{eq:sum} imply \eqref{eq:Dkm}.

Finally, let $n$ be an arbitrary integer with $n>N^2$.
Let $m=\lceil \sqrt{n \log n} \rceil\geq N$ and $k=\lceil \sqrt{n /\log n} \rceil$,
then $km\geq n$, so $N_n(A,M)\leq N_{km}(A,M)$. Applying \eqref{eq:Dkm} for $k,m$
easily yields that
\begin{equation*}
  N_n(A,M)\leq N_{km}(A,M)\leq M^{\alpha \beta n+O(\sqrt{n\log n})}.
\end{equation*}
As $\gamma=\alpha \beta$, inequality \eqref{eq:NAn2} follows. The proof is complete.
\end{proof}


\section{Self-affine functions and the proof of \cref{t:answer}}
\label{s:KK}

The main goal of this section is to prove \cref{t:sa,t:answer}.
First we define a family of self-affine functions $f_{k,m}$, which will be used in Section~\ref{s:generic} as well.

Let $k,m\geq 2$ be fixed integers such that $m$ is odd. For every $i\in
\{0,\dots,m-1\}$ and $j\in \{0,\dots,k-1\}$ define
the one-to-one affine map $F_{ik+j}: \R^2\to \R^2$ as
\[
F_{ik+j}(x,y) = \left(\frac{x+ik+j}{km}, (-1)^{i}\frac{y+j}{k} +
  (i\bmod 2)\right),
\]
see Figure~\ref{f:maps}. As the $F_{ik+j}$ are contractions, Hutchinson's contraction
mapping theorem \cite[Page~713~(1)]{H} implies that there is a unique,
non-empty compact set $K\subset \R^2$ such that
\[
K=\bigcup_{\ell=0}^{km-1} F_{\ell}(K).
\]
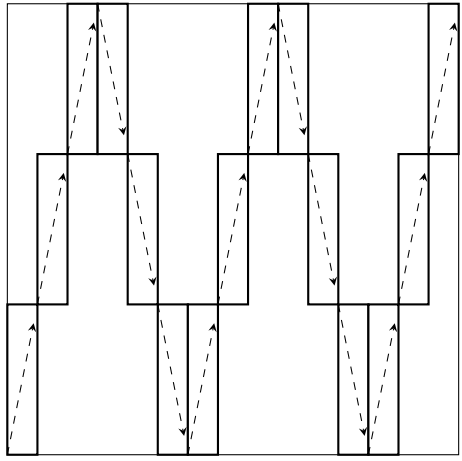
\begin{figure}
\begin{tikzpicture}[>=stealth]
\draw [fill=white] (0,0) rectangle (6,6);

\draw [fill=white, thick] (0,0) rectangle (0.4,2);
\draw [fill=white, thick] (0.4,2) rectangle (0.8,4);
\draw [fill=white, thick]  (0.8,4) rectangle (1.2,6);

\draw [->, dashed] (0,0) -- (0.35,1.75);
\draw [->, dashed] (0.4,2) -- (0.75,3.75);
\draw [->, dashed] (0.8,4) -- (1.15,5.75);

\draw [fill=white, thick] (1.2,4) rectangle (1.6,6);
\draw [fill=white, thick]  (1.6,2) rectangle (2,4);
\draw [fill=white, thick]  (2,0) rectangle (2.4,2);

\draw [->, dashed] (1.2,6) -- (1.55,4.25);
\draw [->, dashed] (1.6,4) -- (1.95,2.25);
\draw [->, dashed] (2,2) -- (2.35,0.25);

\draw [fill=white, thick]  (2.4,0) rectangle (2.8,2);
\draw [fill=white, thick]  (2.8,2) rectangle (3.2,4);
\draw [fill=white, thick]  (3.2,4) rectangle (3.6,6);

\draw [->, dashed] (2.4,0) -- (2.75,1.75);
\draw [->, dashed] (2.8,2) -- (3.15,3.75);
\draw [->, dashed] (3.2,4) -- (3.55,5.75);

\draw [fill=white, thick]  (3.6,4) rectangle (4,6);
\draw [fill=white, thick]  (4,2) rectangle (4.4,4);
\draw [fill=white, thick]  (4.4,0) rectangle (4.8,2);

\draw [->, dashed] (3.6,6) -- (3.95,4.25);
\draw [->, dashed] (4,4) -- (4.35,2.25);
\draw [->, dashed] (4.4,2) -- (4.75,0.25);

\draw [fill=white, thick] (4.8,0) rectangle (5.2,2);
\draw [fill=white, thick]   (5.2,2) rectangle (5.6,4);
\draw [fill=white, thick]   (5.6,4) rectangle (6,6);

\draw [->, dashed] (4.8,0) -- (5.15,1.75);
\draw [->, dashed] (5.2,2) -- (5.55,3.75);
\draw [->, dashed] (5.6,4) -- (5.95,5.75);
\end{tikzpicture}
\caption{For $k=3$ and $m=5$ the illustration shows how the family $\{F_i\}_{i=0}^{14}$ maps $[0,1]^2$ onto rectangles.}
\label{f:maps}
\end{figure}

 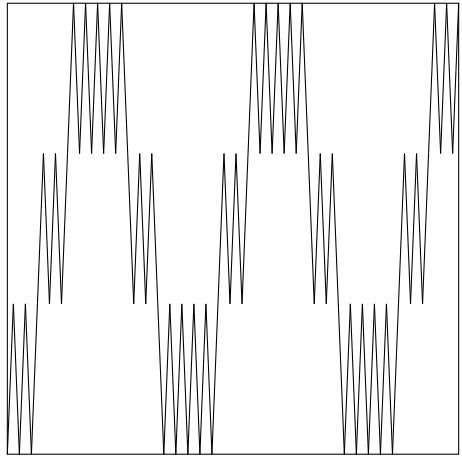
\begin{figure}
 \begin{tikzpicture}
 \draw [fill=white] (0,0) rectangle (6,6);

 \draw [-] (0,0) -- (0.08,2);
 \draw [-] (0.08,2) -- (0.16,0);
 \draw [-] (0.16,0) -- (0.24,2);
 \draw [-] (0.24,2) -- (0.32,0);
 \draw [-] (0.32,0) -- (0.4,2);
 
\draw [-] (0.4,2) -- (0.48,4);
\draw [-] (0.48,4) -- (0.56,2);
\draw [-] (0.56,2) -- (0.64,4);
\draw [-] (0.64,4) -- (0.72,2);
\draw [-] (0.72,2) -- (0.8,4);

\draw [-] (0.8,4) -- (0.88,6);
\draw [-] (0.88,6) -- (0.96,4);
\draw [-] (0.96,4) -- (1.04,6);
\draw [-] (1.04,6) -- (1.12,4);
\draw [-] (1.12,4) -- (1.2,6);

 \draw [-] (2.4,0) -- (2.48,2);
 \draw [-] (2.48,2) -- (2.56,0);
  \draw [-] (2.56,0) -- (2.64,2);
  \draw [-] (2.64,2) -- (2.72,0);
  \draw [-] (2.72,0) -- (2.8,2);
  
  \draw [-] (2.8,2) -- (2.88,4);
  \draw [-] (2.88,4) -- (2.96,2);
  \draw [-] (2.96,2) -- (3.04,4);
  \draw [-] (3.04,4) -- (3.12,2);
  \draw [-] (3.12,2) -- (3.2,4);
  
  \draw [-] (3.2,4) -- (3.28,6);
  \draw [-] (3.28,6) -- (3.36,4);
  \draw [-] (3.36,4) -- (3.44,6);
  \draw [-] (3.44,6) -- (3.52,4);
  \draw [-] (3.52,4) -- (3.6,6);

 \draw [-] (4.8,0) -- (4.88,2);
 \draw [-] (4.88,2) -- (4.96,0);
 \draw [-] (4.96,0) -- (5.04,2);
 \draw [-] (5.04,2) -- (5.12,0);
 \draw [-] (5.12,0) -- (5.2,2);
 
  \draw [-] (5.2,2) -- (5.28,4);
  \draw [-] (5.28,4) -- (5.36,2);
  \draw [-] (5.36,2) -- (5.44,4);
  \draw [-] (5.44,4) -- (5.52,2);
  \draw [-] (5.52,2) -- (5.6,4);
 
 \draw [-] (5.6,4) -- (5.68,6);
 \draw [-] (5.68,6) -- (5.76,4);
 \draw [-] (5.76,4) -- (5.84,6);
 \draw [-] (5.84,6) -- (5.92,4);
 \draw [-] (5.92,4) -- (6,6);

\draw [-] (1.2,6) -- (1.28,4);
\draw [-] (1.28,4) -- (1.36,6);
\draw [-] (1.36,6) -- (1.44,4);
\draw [-] (1.44,4) -- (1.52,6);
\draw [-] (1.52,6) -- (1.6,4);

\draw [-] (1.6,4) -- (1.68,2);
\draw [-] (1.68,2) -- (1.76,4);
\draw [-] (1.76,4) -- (1.84,2);
\draw [-] (1.84,2) -- (1.92,4);
\draw [-] (1.92,4) -- (2,2);

\draw [-] (2,2) -- (2.08,0);
\draw [-] (2.08,0) -- (2.16,2);
\draw [-] (2.16,2) -- (2.24,0);
\draw [-] (2.24,0) -- (2.32,2);
\draw [-] (2.32,2) -- (2.4,0);

\draw [-] (3.6,6) -- (3.68,4);
\draw [-] (3.68,4) -- (3.76,6);
\draw [-] (3.76,6) -- (3.84,4);
\draw [-] (3.84,4) -- (3.92,6);
\draw [-] (3.92,6) -- (4,4);

\draw [-] (4,4) -- (4.08,2);
\draw [-] (4.08,2) -- (4.16,4);
\draw [-] (4.16,4) -- (4.24,2);
\draw [-] (4.24,2) -- (4.32,4);
\draw [-] (4.32,4) -- (4.4,2);

\draw [-] (4.4,2) -- (4.48,0);
\draw [-] (4.48,0) -- (4.56,2);
\draw [-] (4.56,2) -- (4.64,0);
\draw [-] (4.64,0) -- (4.72,2);
\draw [-] (4.72,2) -- (4.8,0);

 \end{tikzpicture}
\caption{The third level approximation of the self-affine function $f_{3,5}$.}
\label{f:new}
\end{figure}

It is easy to see that $K$ is a graph of a function $f_{k,m}\colon [0,1]\to
[0,1]$, and $f_{k,m}$ can be approximated as follows.
Let $D=\{(x,x): x\in [0,1]\}$ be the diagonal of $[0,1]^2$ and define
$\Gamma \colon \iP(\R^2) \to \R$ as
\[\Gamma(A)=\bigcup_{\ell=0}^{km-1} F_{\ell}(A).\]
If $\Gamma^{i}$ denotes the $i^{\textrm{th}}$ iterate $\Gamma \circ \dots \circ \Gamma$, then $\Gamma^{i}(D)$ is the graph of a piecewise linear function $f^{i}_{k,m}$ which converges uniformly to $f_{k,m}$ as $i\to \infty$. Clearly $f_{k,m}$ is a self-affine function with $f_{k,m}(0)=0$ and $f_{k,m}(1)=1$, and the definition yields that $f_{k,m}$ is H\"older continuous with exponent $\log k/\log (km)$. Figure~\ref{f:new} shows the piecewise linear function $f^3_{3,5}$, which approximates $f_{3,5}$.

Define the set
\[
\Delta=\left\{\frac{\log k}{\log (km)}: k,m\geq 2 \textrm{ are integers and $m$ is odd}\right\}.
\]
Then $\Delta$ is a countable dense subset of $(0,1)$, since every rational $p/q\in(0,1)$ is in $\Delta$ by
$k=3^p$ and $m=3^{q-p}$. For all
$\alpha\in \Delta$ fix $k,m$ such that $\alpha=\log k/\log (km)$,
and define $f_\alpha=f_{k,m}\in C^{\alpha}[0,1]$.

\begin{proof}[Proof of \cref{t:sa}]
Fix $\alpha=\log k/\log (km)\in \Delta$ such that
$f_\alpha=f_{k,m}$. We use the scaled local times with $M=km$.
By \cref{t:Holder,t:variation} it is enough to show that
$f_{\alpha} \in \iA_n(\alpha, M)$ for all $n\geq 2$.
Let us fix $n\geq 2$. Clearly $M^{1-\alpha}=m$, and the construction of $f_\alpha$ yields
that for every $\ell \leq n$, $p< M^{\ell}$, and $0\leq q<k^n$ we have
\begin{equation*}
|I\in \iI_{n,\ell,p}: ~\exists x\in I,~f_{\alpha}(x)\in (qk^{-n},(q+1)k^{-n})| =
 m^{n-\ell}=M^{(1-\alpha)(n-\ell)}.
\end{equation*}
Similarly,
\begin{equation*}
  |I\in \iI_{n,\ell,p}: ~qk^{-n}\in  f_{\alpha}(I)| \leq 2 m^{n-\ell} =
  2M^{(1-\alpha)(n-\ell)}.
\end{equation*}
The above and $M^{-\alpha}=1/k$ yield that for all $\ell,p,q$ we have
\begin{equation*}
A_{n,\ell,p,q}(f_{\alpha})\leq 3M^{(1-\alpha)(n-\ell)}.
\end{equation*}
Thus $f_{\alpha} \in \iA_n(\alpha, M)$, and the proof is complete.
\end{proof}

\begin{remark}
Note that the bound we get for $A_{n,\ell,p,q}(f_{\alpha})$ does not use the $n^2$
factor. It is possible to go through the proof \cref{t:Holder} with such a
stronger assumption, which would slightly improve the bounds on $N_n(A,M)$
with $O(\log n)$ in place of $O(\log^2 n)$ error term.
\end{remark}

\begin{proof}[Proof of \cref{t:answer}]
Let $0<\alpha<1$. By \cref{t:KKM,t:M,t:A} it is enough to prove that
\begin{align*}
\overline{H}(\alpha,\beta)&\leq 1-\beta \textrm{ for every } \alpha<\beta\leq 1; \\
\overline{V}(\alpha,\beta)&\leq \max\left\{\alpha\beta, \frac{\beta}{\beta+1}\right\} \textrm{ for all } \beta>0.
\end{align*}

For the first inequality let $\gamma\in\Delta\cap (\alpha,\beta)$ be arbitrary,
then $f_\gamma\in C^{\alpha}[0,1]$. Suppose that $f_{\gamma}$
is $\beta$-H\"older continuous on some set $A\subset [0,1]$.
Theorem~\ref{t:sa}~\eqref{eq:sa1} yields that $\overline{\dim}_\iM A\leq 1-\gamma$, thus
$\overline{H}(\alpha,\beta) \leq 1-\gamma$. Since $\Delta$ is dense in $(\alpha,\beta)$, we obtain that
$\overline{H}(\alpha,\beta)\leq 1-\beta$.

For the second inequality define $\delta=\max\{\alpha,1/(\beta+1)\}$ and let $\gamma\in \Delta\cap (\delta,1)$ be arbitrary.
Then $f_\gamma\in C^{\alpha}[0,1]$. Assume that $f_{\gamma}$ has finite $\beta$-variation on some
$A$. Theorem~\ref{t:sa}~\eqref{eq:sa2} and $\gamma>\delta$ imply that
$\overline{\dim}_\iM A\leq \max\{1-\gamma,\gamma \beta\}=\gamma \beta$, so
$\overline{V}(\alpha,\beta)\leq \gamma \beta$. As $\Delta$ is dense in $(\delta,1)$,
we have $\overline{V}(\alpha,\beta)\leq \delta \beta=\max\{\alpha \beta,\beta/(\beta+1)\}$.
\end{proof}

\section{Restrictions of fractional Brownian motion} \label{s:fBm}

Let $0<\alpha<1$ be fixed and let $B\colon [0,1]\to \R$ be a fractional Brownian motion of Hurst index $\alpha$. We think of $B$ as
a random function from $[0,1]$ to $\R$.

The main goal of this section is to prove Theorem~\ref{t:fBm}.
By Theorems~\ref{t:Holder} and \ref{t:variation} it is enough to prove the following proposition. For the notation $\iA_n(\alpha,2)$ see Definition~\ref{d:An}.

\begin{proposition}\label{p:key}
Let $0<\alpha<1$ and let $B\colon [0,1]\to \R$ be a fractional Brownian motion of Hurst index $\alpha$. Then, almost surely, $B\in
  \iA_n(\alpha,2)$ for all $n$ large enough.
\end{proposition}

First we define a discrete (truncated) scaled local time.

\begin{definition}
For all $n\in \N$ let
\[\iE_n=\{i2^{-n}: 0\leq i\leq 2^{n}-1\}.\]
For all $0\leq m\leq n$ let
\[\iL_{n,m,p}=I_{m,p}\cap \iE_n.\]
Clearly, $|\iL_{n,m,p}|=2^{n-m}$.
For a function $f\colon [0,1]\to \R$ the \emph{discrete scaled local time} $S_{n,m,p,q}(f)$ is the number of
points in $\iL_{n,m,p}$ which are mapped to $J_{n,q}$ by $f$:
\[
S_{n,m,p,q}(f)= \left|\{x\in\iL_{n,m,p}: f(x)\in J_{n,q}\}\right|.
\]
For every $n\in \N^+$ define
\[\iS_n(\alpha)=\{f: S_{n,m,p,q}(f)\leq (n \log n) 2^{(1-\alpha)(n-m)} ~ \forall m\leq n ,\, p< 2^{n},\, |q|\leq n2^{\alpha n}\}.\]
\end{definition}

First we need to prove some lemmas. To avoid technical difficulties we assume that the domain of $B$ is extended to $[0,\infty)$ when necessary. 

\begin{definition} Let $(\Omega, \iF, \P)$ be the probability space on which our fractional Brownian motion is defined, and let
$\iF_t=\sigma(B(s): 0\leq s\leq t)$ be the natural filtration. If $\tau\colon \Omega \to [0,\infty]$ is a stopping time then define the $\sigma$-algebra
\[\iF_{\tau}=\{A\in \iF: A\cap \{\tau\leq t\}\in \iF_t \textrm{ for all } t\geq 0\}.\]
For all stopping times $\tau$ and integers $0\leq m\leq n$ and $q$ let
\[X^{n,m,q}_{\tau}=|\{k\in \{1,\dots,2^{n-m}\}: B(\tau+k2^{-n})\in J_{n,q}\}|.\]
\end{definition}

\begin{lemma} \label{l:sln} There is a finite constant $c=c(\alpha)$ depending only on $\alpha$ such that for every bounded
stopping time $\tau$ and integers $0\leq m\leq n$ and $q$ we have, almost surely,
\[ \E(X^{n,m,q}_{\tau} \, | \, \iF_{\tau})\leq c 2^{(1-\alpha)(n-m)}.\]
\end{lemma}

\begin{proof} Pitt \cite[Lemma~7.1]{P} showed that the property of
\emph{strong local nondeterminism} holds for fractional Brownian motion, that is, there is a constant
$c_1=c_1(\alpha)>0$ such that for all $t\geq 0$, almost surely,
\begin{equation} \label{eq:tau} \Var(B(\tau+t) \, | \, \iF_{\tau})\geq c_1 t^{2\alpha}.\end{equation}
Let us fix $t>0$. As $B$ is Gaussian, almost surely the conditional distribution $B(\tau+t)\,|\,\iF_{\tau}$ is normal,
and \eqref{eq:tau} implies that its density function is bounded by $1/(\sqrt{c_1}t^{\alpha})$. Therefore, almost surely,
\begin{equation} \label{eq:tau2}
\P(B(\tau+t)\in J_{n,q}  \, | \, \iF_{\tau})\leq \int_{q2^{-\alpha n}}^{(q+1)2^{-\alpha n}}
\frac{1}{\sqrt{c_1}t^{\alpha}} \,\mathrm{d} x=c_2(t2^{n})^{-\alpha},
\end{equation}
where $c_2=1/\sqrt{c_1}$. Applying \eqref{eq:tau2} for finitely many $t$ implies that
\begin{align*} \E(X^{n,m,q}_{\tau} \, | \, \iF_{\tau})&=\sum_{k=1}^{2^{n-m}} \P(B(\tau+k2^{-n})\in J_{n,q}  \, | \, \iF_{\tau}) \\
&\leq \sum_{k=1}^{2^{n-m}} c_2 k^{-\alpha}\leq c_2\int_{0}^{2^{n-m}} x^{-\alpha}\,\mathrm{d} x=c 2^{(1-\alpha)(n-m)},
\end{align*}
where $c=c_2/(1-\alpha)$. The proof is complete.
\end{proof}

\begin{lemma} \label{l:Snmpq} There is a finite constant $C=C(\alpha)$ depending only on $\alpha$
such that for all $m,n,p,q$ and $\ell\in \N^+$ we have
\[\P(S_{n,m,p,q}(B)\geq  \ell C 2^{(1-\alpha)(n-m)})\leq 2^{-\ell}.\]
\end{lemma}

\begin{proof}  Let $c$ be the constant in Lemma~\ref{l:sln}, clearly we may assume that $c\geq 1$.
We will show that $C=3c$ satisfies the lemma.
We define stopping times $\tau_0,\dots, \tau_\ell$. Let $\tau_0=0$. If $\tau_{k}$ is defined for some $0\leq k<\ell$ then
let $\tau_{k+1}$ be the first time such that the contribution of $B$ to $S_{n,m,p,q}(B)$ on $(\tau_k,\tau_{k+1}]\cap \iL_{n,m,p}$
is at least $2c 2^{(1-\alpha)(n-m)}$ if such a time exists, otherwise let $\tau_{k+1}=1$.
Then $c\geq 1$ and the definition of stopping times yield that
\begin{align*} \P(S_{n,m,p,q}(B)\geq  3\ell c 2^{(1-\alpha)(n-m)})&\leq \P(S_{n,m,p,q}(B)\geq  \ell(2c 2^{(1-\alpha)(n-m)}+1)) \\
&\leq \P(\tau_\ell<1)=\prod_{k=1}^{\ell} \P(\tau_{k}<1 \,|\, \tau_{k-1}<1).
\end{align*}
Note that we may assume that $\P(\tau_\ell<1)>0$ and hence the above conditional probabilities are defined,
otherwise we are done immediately. Therefore it is enough to prove that $\P(\tau_{k}<1 \,|\, \tau_{k-1}<1)\leq 1/2$ for all $1\leq k\leq \ell$.
The definition of $X^{n,m,q}_{\tau}$, Lemma~\ref{l:sln}, and the conditional Markov's inequality imply that, almost surely,
\[\P(\tau_k<1 \,|\, \iF_{\tau_{k-1}})\leq \P(X^{n,m,q}_{\tau_{k-1}}\geq 2c 2^{(1-\alpha)(n-m)}\, | \, \iF_{\tau_{k-1}})\leq 1/2.\]
Therefore the tower property of conditional expectation yields that
\[\P(\tau_{k}<1 \,|\, \tau_{k-1}<1)\leq 1/2,\]
which completes the proof.
\end{proof}

The following lemma is a discrete version of Proposition~\ref{p:key}.

\begin{lemma}\label{l:dkey}
  Almost surely, $B\in \iS_n(\alpha)$ for all large enough $n$.
\end{lemma}

\begin{proof} We give an upper bound for $\P(B\notin \iS_n(\alpha))$ by applying Lemma~\ref{l:Snmpq}
with $\ell=\lfloor (n/C) \log n \rfloor$ to each of the relevant $m,p,q$. Since $0\leq m\leq n$, $0\leq p\leq 2^m-1$ and
$|q|\leq n2^{\alpha n}$, there are at most $(n+1)2^n (2n2^{\alpha n}+1)$ possibilities to choose $m,p,q$. Therefore
Lemma~\ref{l:Snmpq} implies that
\[\P(B\notin \iS_n(\alpha))\leq (n+1)2^n (2n2^{\alpha n}+1)2^{-\lfloor (n/C) \log n \rfloor}=2^{-(n/C)\log n+O(n)}.\]
Therefore $\sum_{n=1}^{\infty} \P(B\notin \iS_n(\alpha))<\infty$, so the Borel--Cantelli lemma implies that
\[\P\left(B\in \liminf_{n}\iS_n(\alpha)\right)=1. \qedhere \]
\end{proof}

The following lemma is well known, see the more general \cite[Corollary~7.2.3]{MR}.

\begin{lemma} \label{l:ub} Almost surely, we have
\[\limsup_{h\to 0+} \sup_{0\leq t\leq 1-h} \frac{|B(t+h)-B(t)|}{\sqrt{2h^{2\alpha} \log(1/h)}}\leq 1.\]
\end{lemma}

Now we are ready to prove Proposition~\ref{p:key}.

\begin{proof}[Proof of Proposition~\ref{p:key}]
By Lemmas~\ref{l:dkey} and \ref{l:ub} we can choose a random $N\in \N^+$ such that, almost surely,
for every $n\geq N$ we have the following properties:
\begin{enumerate}[label={\textup{(\roman*)}},nosep]
\item \label{i1} $\max_{t\in [0,1]} |B(t)|<N-1$;
\item \label{i2} $B\in \iS_n(\alpha)$;
\item \label{i3} $\diam B(I_{n,p})\leq 2\sqrt{n} 2^{-\alpha n}$ for all $0\leq p\leq 2^{n}-1$;
\item \label{i4} $(4\sqrt{n}+3)(n\log n)\leq n^2$.
\end{enumerate}
Fix a path of $B$ for which the above properties hold. Let us fix an arbitrary $n\geq N$,
it is enough to prove that $B\in \iA_n(\alpha,2)$. Let
$0\leq m\leq n$, $0\leq p\leq 2^m-1$ and $q\in \Z$ be given, we need to show that
\begin{equation} \label{eq:Bnm} A_{n,m,p,q}(B)\leq n^2 2^{(1-\alpha) (n-m)}.
\end{equation}
Property \ref{i1} yields that if $q'\in \Z$ with $|q'|>n2^{\alpha n}$ then $S_{n,m,p,q'}(B)=0$.
Therefore \ref{i2} implies that for all $q'\in \Z$ we have
\begin{equation} \label{eq:Snm} S_{n,m,p,q'}(B)\leq (n\log n) 2^{(1-\alpha) (n-m)}.
\end{equation}
Let $I_{n,p'}$ be a time interval of order $n$ such that $I_{n,p'}\subset I_{m,p}$ and $B(I_{n,p'})\cap J_{n,q}\neq \emptyset$,
then \ref{i3} yields that
\begin{equation} \label{eq:BInp} B(I_{n,p'})\subset \bigcup_{q': |q'-q|\leq 2\sqrt{n}+1} J_{n,q'}.
\end{equation}
Finally, \eqref{eq:BInp}, \eqref{eq:Snm} and \ref{i4} imply that
\begin{align*}  A_{n,m,p,q}(B)&\leq \sum_{q': |q'-q|\leq 2\sqrt{n}+1} S_{n,m,p,q'}(B) \\
&\leq (4\sqrt{n}+3) (n\log n) 2^{(1-\alpha) (n-m)} \\
&\leq n^2 2^{(1-\alpha) (n-m)}.
\end{align*}
Hence \eqref{eq:Bnm} holds, and the proof of Proposition~\ref{p:key} is complete.
\end{proof}

\subsection{Dimension of the record times}

We include here a proof of Proposition~\ref{p:R}, which we could not
find in the literature. Recall that $\{B(t): t\in [0,1]\}$ is a fractional
Brownian motion of Hurst index $\alpha$.
The lower bound is quite elementary, while the upper bound relies on a first
moment computation and on a result of Molchan.
He studied the distribution of the maximal value of fractional Brownian motion on $[0,t]$
and of the time $\tau_{\max}(t)$ when it is achieved.
Specifically, Molchan \cite[Theorem~2]{Mo} proved the following.

\begin{theorem}
$\P(\tau_{\max}(1) < x) = x^{1-\alpha+o(1)}$ as $x\to0$.
\end{theorem}

\begin{proof}[Proof of Proposition~\ref{p:R}]
With probability one $B$ is $\gamma$-H\"older continuous for every $\gamma<\alpha$ and maps $\iR$ to the non-degenerate interval $I=[0,\max \{B(t): t\in [0,1]\}]$, thus 
$\dim_\iH \iR \geq \alpha \dim_\iH I=\alpha$.

Therefore it is enough to prove that $\overline{\dim}_\iM \iR\leq \alpha$ almost surely. Using that
$\tau_{\max}(t)$ and $t \tau_{\max}(1)$ have the same distribution and $\{B(1-t)-B(1): t\in [0,1]\}$ is also a fractional
Brownian motion, for all $0\leq \eps<t\leq 1$ we obtain
\begin{align*} \P(\iR \cap [t-\eps,t] \neq\emptyset)&=\P(\tau_{\max}(t) \geq t-\eps)=\P(\tau_{\max}(1)\geq 1-\eps/t) \\
&=\P(\tau_{\max}(1)\leq \eps/t)=(\eps/t)^{1-\alpha+o(1)},
\end{align*}
with the $o(1)$ term tending to $0$ as $\eps/t\to 0$.

Let $N(m)$ be the number of intervals $[(i-1)/m,i/m]$ which intersect $\iR$. Let $\delta>0$ be arbitrary and fix $s\in \N^+$ such that
  $\P(\iR \cap [t-\eps,t] \neq\emptyset) \leq (\eps/t)^{1-\alpha-\delta}$ whenever $\eps/t<1/s$. Using this above for $i>s$ and the trivial bound $1$ for $i\leq s$, for every large enough $m$ we obtain
  \[
  \E N(m) \leq s + \sum_{i=s+1}^{m} (1/i)^{1-\alpha-\delta}
  \leq s + Cm^{\alpha+\delta}
  \leq 2Cm^{\alpha+\delta},
  \]
where $C$ is a finite constant depending only on $\alpha+\delta$. By Markov's inequality
  \[
  \P(N(m) > m^{\alpha+2\delta}) \leq 2C m^{-\delta}
  \]
  for any $m$ large enough. Applying this for $m=2^n$ yields that
\[\sum_{n=1}^{\infty} \P(N(2^n)>2^{n(\alpha+2\delta)})<\infty.\]
Thus the Borel--Cantelli lemma implies that, almost surely,
$N(2^{n}) \leq 2^{n(\alpha+2\delta)}$ for all large enough $n$.
Therefore $\overline \dim_\iM \iR \leq \alpha+2\delta$. As $\delta>0$ was arbitrary,
$\overline \dim_\iM \iR \leq \alpha$ almost surely.
\end{proof}

\section{Higher dimensional Brownian motion} \label{s:2D}

The aim of this section is to prove Theorem~\ref{t:2DBM}. The idea is to
find (in a greedy manner) large sets along which a simple random walk in
$\Z^2$ is monotone. Since the scaled simple random walk converges to
Brownian motion, this gives sets along which $B$ is monotone. To control
the dimension of the limit sets we estimate the energy of the discrete
sets and apply a version of Frostman's lemma to bound the dimension.

Given a simple random walk $S\colon \N\to \Z^2$, define the \emph{greedy increasing subset} by
\[
a_0=0 \quad \textrm{and} \quad a_{i+1}=\min\{a>a_i: S(a)-S(a_i)\in \Z^2_+\}.
\]
Our first task is to prove tightness for the number and structure of record
times in $[0,n)$. Since our argument may apply in similar situations, we
state some of our arguments in the more general context of sums of i.i.d.\
variables with power law tails.

Before focusing on the case of random walks, we prove Theorem~\ref{t:acc},
a limit theorem, which will allow us to transfer estimates from the random
walk setting to Brownian motion. As Theorem~\ref{t:acc} below is a quite
general result about random sequences with i.i.d.\ increments, we hope that
it will find further applications.

\subsection{Energy of renewal processes}

Fix $0<\alpha<1$. Let $\tau\in\N^+$ be some random variable, and assume that
there are $c_1,c_2\in \R^+$ such that for all $n\in \N^+$ we have
\begin{equation} \label{eq:Xn}
  c_1n^{-\alpha} \leq \P(\tau>n) \leq c_2 n^{-\alpha}.
\end{equation}
Let $\{\tau_i\}_{i\geq 1}$ be an i.i.d.\ sequence with the law of $\tau$.
Define 
\[T_k=\sum_{i=1}^k \tau_i \quad  \textrm{and} \quad  \iT = \{T_k: k\geq 1\}.\]
The number of steps before reaching $n$ is denoted by $m_n = |\iT \cap
[0,n)|$.

The following lemma is fairly standard.

\begin{lemma}\label{L:m_n_tight}
  There are constants $c_3,c_4$ such that for all $t,n>0$ we have
  \begin{enumerate}[label={\textup{(\roman*)}},nosep]
  \item \label{cl1} $\P(m_n < tn^{\alpha}) \leq c_3t$.
  \item \label{cl2} $\P(m_n > tn^{\alpha}) \leq e^{-c_1t}$.
  \item \label{cl3} $\E m_n \leq c_4 n^{\alpha}$, and more generally, for all integers $i < j < k$ we have
    \[
    \E \big(|\iT \cap [j,k)| \big| i\in\iT\big) \leq c_4 (k-j)^\alpha.
    \]
\end{enumerate}
\end{lemma}

\begin{proof}
  Claim \ref{cl1} is given by \cite[Lemma~4.2]{ABP}. Inequality~\eqref{eq:Xn} and $(1-u) \leq e^{-u}$ imply that
  \begin{align*}
    \P(m_n> tn^{\alpha})
    &\leq \P(\tau_i\leq n \textrm{ for all } i\leq \lceil tn^{\alpha} \rceil)
    \leq \left(\P(\tau \leq n)\right)^{tn^{\alpha}} \\
    &\leq \left(1-c_1 n^{-\alpha}\right)^{tn^\alpha}
    \leq e^{-c_1 t},
  \end{align*}
so \ref{cl2} holds. The first bound of \ref{cl3} follow easily from \ref{cl2}.
The general bound holds since the first $\ell$ with $T_\ell \in[j,k)$ (if
there is such $\ell$) is a stopping time. Applying \ref{cl2} to the sequence
starting at time $\ell$ completes the proof.
\end{proof}

\begin{definition}
 Let $\mu$ be a non-atomic \emph{mass distribution} on a metric space $(E,\rho)$, (that is, a Borel measure on $E$ with $0<\mu(E)<\infty$). For $\gamma>0$, define the \emph{$\gamma$-energy of $\mu$} by
  \[
  \iE_\gamma(E,\mu) = \iint_{E^2}
  \frac{\mathrm{d} \mu(x)\,\mathrm{d} \mu(y)}{\rho(x,y)^{\gamma}}.
  \]
\end{definition}

For the following theorem see \cite[Theorem~4.27]{MP}.

\begin{theorem} \label{t:energy} Let $\mu$ be a non-atomic
mass distribution on a metric space $E$ with $\iE_\gamma(E,\mu)<\infty$.
Then $\dim_\iH E\geq \gamma$.
\end{theorem}

Consider the set $S_n = \big(\iT + [0,1)\big) \cap [0,n)$, endowed
with Lebesgue measure $\lambda$. We next estimate the $\gamma$-energy of
$\lambda$.

\begin{lemma}\label{L:energy_small}
Let $0<\gamma<\alpha$. There is a finite constant $c(\gamma)$,
such that for all $n$ we have
\[\E \iE_\gamma(S_n,\lambda) \leq c(\gamma) n^{2\alpha-\gamma}.\]
\end{lemma}

\begin{proof} The argument is to consider the contribution to the energy
from pairs $x,y$ with distance at various scales, and the largest scale will
dominate the rest.

Up to a factor of $2$ we may restrict the integral to $x<y$. We split the integral on $S_n\times S_n$ into several parts.
Note that $S_n$ is a disjoint union of unit intervals.
Let $P_0$ be the contribution to $\iE_\gamma(S_n,\lambda)$ from pairs $x\in[i,i+1)$ and $y\in [j,j+1)$ where $i,j\in \iT$ and
$0\leq j-i\leq 1$. The number of such pairs $\{i,j\}$ is at most $m_n$, so Lemma~\ref{L:m_n_tight}~\ref{cl3} yields that
\[
\E P_0\leq (\E m_n) \int_0^1 \int_x^2 |x-y|^{-\gamma} dy\, dx =O_{\gamma}(n^\alpha).
\]

For $k\geq 1$ let $P_k$ be the contribution to $\iE_\gamma(S_n,\lambda)$ from pairs $x\in[i,i+1)$ and $y\in [j,j+1)$
where $i,j\in \iT$ and $i+2^{k-1}< j \leq i+2^k$. Let $M_k$ denote the number of such pairs
$\{i,j\}$. For such $\{i,j\}$
the contribution from $x\in[i,i+1)$ and $y\in [j,j+1)$ to $P_k$ is at most
\[
\int_i^{i+1} \int_j^{j+1} |x-y|^{-\gamma} dy\, dx \leq 2^{-(k-1)\gamma},
\]
where we used $y-x\geq 2^{k-1}$. Thus $P_k \leq 2^{-(k-1)\gamma} M_k$. Lemma~\ref{L:m_n_tight}~\ref{cl3} yields that
for every $i$, conditioned on $i\in\iT$, the expected number of $j$ in $\iT \cap (i+2^{k-1},i+2^k]$ is at most
$c_4 2^{(k-1)\alpha}$. Lemma~\ref{L:m_n_tight}~\ref{cl3} also implies that the expected number of $i<n$ in $\iT$ equals $\E m_n\leq c_4 n^\alpha$. Therefore $\E M_k \leq c_4^2 n^\alpha 2^{(k-1)\alpha}$ and we obtain that
\[
\E P_k \leq c_4^2 n^\alpha 2^{(k-1)(\alpha-\gamma)}.
\]

This partition gives the
identity $\iE_\gamma(S_n,\lambda) =2\sum_{k=0}^{\infty} P_k$. As $P_k=0$ whenever $2^{k-1}>n$, we
have $\E \iE_\gamma(S_n,\lambda)= 2\sum_{2^k\leq 2n} \E P_k$. With the bounds above, the largest $k$ dominates the sum and we
arrive at the inequality $\E \iE_\gamma(S_n,\lambda) \leq c(\gamma) n^{2\alpha-\gamma}$.
\end{proof}

We will wish to work with rescaled sets. For all $n\in \N^+$ let
\[D_n = \left(\frac1n \iT \right) \cap [0,1) \quad \textrm{and} \quad
C_n = \frac1n S_n = D_n + [0,1/n).\]
Define the measure $\mu_n=n^{1-\alpha} \lambda|_{C_n}$, that is, $n^{1-\alpha}$ times the Lebesgue
measure restricted to $C_n$.

\begin{lemma}\label{L:mu_tight}
Let $0<\gamma<\alpha$ and $\eps>0$. Then there exist $N,N_{\gamma}\in \N^+$ depending on $\eps$ so that with probability at least $1-\eps$ we have
\[
N^{-1} \leq \mu_n([0,1]) \leq N \quad \text{ and } \quad \iE_\gamma(C_n,\mu_n) \leq N_{\gamma}.
\]
\end{lemma}

\begin{proof}
Let $I=[0,1]$ and let $N\in \N^+$ be arbitrary. For all $n\in \N^+$ the following three inequalities hold.
Markov's inequality and $\E \mu_n(I)=1$ yield that
\begin{equation*} \label{eq:A1} \P(\mu_n(I)>N)\leq \frac{\E \mu_n(I)}{N}=\frac{1}{N}.
\end{equation*}
Lemma~\ref{L:m_n_tight}~\ref{cl3} and \ref{cl1} yield that
\begin{equation*}  \P(\mu_n(I)<N^{-1})=\P(m_n<N^{-1}\E m_n)
\leq \P(m_n<N^{-1}c_4n^{\alpha})\leq \frac{c_3 c_4}{N}.
\end{equation*}
Lemma~\ref{L:energy_small} yields that $\E \iE_{\gamma}(C_n,\mu_n) \leq c(\gamma)$.
Indeed, $\iE_\gamma(C_n,\lambda) = n^{\gamma-2} \iE_{\gamma}(S_n,\lambda)$ and
since $\mu_n = n^{1-\alpha}\lambda$ on $C_n$, changing to $\mu_n$ gives a
further factor of $n^{2-2\alpha}$. Now Markov's inequality implies that
\begin{equation*} \P(\iE_\gamma(C_n,\mu_n) > N_{\gamma}) \leq \frac{c(\gamma)}{N_{\gamma}}.
\end{equation*}
The above three inequalities with large enough $N,N_{\gamma}$ complete the proof.
\end{proof}

\subsection{A limit theorem}
Theorem~\ref{t:acc} is concerned with a sequence of sequences of
i.i.d.\ variables $\{\tau^{(n)}_i\}_{i\geq 1}$ satisfying
\eqref{eq:Xn}. That is, for each fixed $n$ the variables
$\{\tau^{(n)}_i\}_{i\geq 1}$ are i.i.d., but there could be arbitrary
dependencies between variables with different numbers $n$. The superscript $n$ is also the parameter used for scaling sums of the $n$th sequence. Thus we
denote 
\[T^{(n)}_k=\sum_{i=1}^k \tau^{(n)}_i \quad \textrm{and} \quad \iT^{(n)}=\{T^{(n)}_k: k\geq 1\}.\] 
For all $n\in \N^+$ define
\[
D_n = \left(\frac 1n \iT^{(n)}\right)\cap [0,1) \quad \textrm{and} \quad C_n=D_n+[0,1/n).
\]

\begin{theorem}\label{t:acc}
With the notations above, almost surely, $\{D_n\}_{n\geq 1}$ has an
accumulation point $D$ in the Hausdorff metric such that $\dim_\iH D \geq \alpha$.
\end{theorem}

\begin{proof}
Let $\mu_n$ be $n^{1-\alpha}$ times the Lebesgue measure on $C_n$.
Fix $\eps>0$. For some finite constants $\{N_{\gamma}\}_{0\leq \gamma<\alpha}$ for all $n$ define the event
\[
\iB_n =\{N_0^{-1} \leq \mu_n([0,1]) \leq N_0 \text{ and }
\iE_{\gamma}(C_n,\mu_n) \leq N_{\gamma} \textrm{ for all } 0<\gamma<\alpha\}.
\]
Since the map $\gamma\mapsto \iE_{\gamma}(C_n,\mu_n)$ is non-decreasing, applying Lemma~\ref{L:mu_tight} for a sequence of parameters $\gamma_k\nearrow \alpha$ with $\eps_k=2^{-k}\eps$ implies that there are constants $N_{\gamma}$ such that
$\P(\iB_n) \geq 1-\eps$ for all $n$.
Let $\iB = \limsup_{n} \iB_n$, then $\P(\iB)\geq 1-\eps$. Since $\eps>0$ was arbitrary,
it is enough to prove that the theorem is satisfied whenever $\iB$ holds.

Assume that $\iB$ holds, then there is a random subsequence
$\{n_i\}_{i\geq 1}$ such that the events $\iB_{n_i}$ hold for all
$i\in \N^+$. Since $N_0^{-1}\leq \mu_{n_i}([0,1])\leq N_0$ for all $i$, by passing
to a subsequence we may assume that $\mu_{n_i}\to \mu$ weakly, where
$\mu$ is a measure on $[0,1]$ satisfying $N_0^{-1}\leq \mu([0,1])\leq N_0$.
Similarly, we may assume that $D_{n_i}\to D$ in the Hausdorff metric for
some compact set $D\subset [0,1]$. As $C_n = D_n + [0,1/n]$ is close
in the Hausdorff metric to $D_n$, we get $\supp(\mu)\subset D$. For all
$0<\gamma<\alpha$ we obtain
\[
\iE_{\gamma}(D,\mu)\leq \liminf_i \iE_{\gamma}(\mu_{n_i}) \leq N_{\gamma}<\infty,
\]
for the first inequality see e.g.\ \cite[Lemma~2.2]{M:Borel}.
Theorem~\ref{t:energy} now implies that $\dim_\iH D\geq \alpha$,
and the proof is complete.
\end{proof}

\subsection{Application to random walks}

We now apply Theorem~\ref{t:acc} to random walks on $\Z^2$ and thus prove
Theorem~\ref{t:2DBM}.

For each $n$, let $S^{(n)}$ be a simple random walk on $\Z^2$, and define
the rescaled random walks by $W_n(t) = \sqrt{2} n^{-1/2} S^{(n)}(\lfloor nt
\rfloor)$. It is well known that it is possible to construct the walks
$S^{(n)}$ and two-dimensional Brownian motion $\{B(t): t\in [0,1]\}$ on the
same probability space so that $W_n$ converges uniformly to $B$ on the
interval $[0,1]$, see e.g.\ \cite[Theorem~3.5.1]{LL} or \cite{MP}. We
henceforth assume such a coupling.

Recall that for each walk we construct the \emph{greedy increasing subset}
by
\[
a^{(n)}_0=0 \quad \textrm{and} \quad a^{(n)}_{i+1}=\min\{a>a^{(n)}_i: S^{(n)}(a)-S^{(n)}(a^{(n)}_i)\in \Z^2_+\}.
\]
For every $n$, this sequence has i.i.d.\ increments with the law of
\[
\tau = \inf \{k>0: S(k) \in \Z^2_+\}.
\]
We use the notation $a_n\sim b_n$ if $a_n/b_n \to 1$ as $n\to \infty$. We need the following known estimate.

\begin{theorem}\label{t:greedy}
  Let $S\colon \N\to \Z^2$ be a two-dimensional simple random walk. Let $\tau$ be the hitting time of the positive quadrant: $\tau=\inf\{k>0 :
  S(k) \in \Z_+^2\}$. Then there is a $c\in \R^+$ so that \[\P(\tau>n) \sim c n^{-1/3}.\]
\end{theorem}

For the above theorem see the general result of Denisov and Wachtel
\cite[Theorem~1]{DW}, or a bit weaker one due to Varopoulos \cite[(0.3.3)
and (0.4.1)]{V}. On exit times of planar Brownian motion from cones see
Evans \cite[Corollary~5(i)]{Ev} or the somewhat weaker
\cite[Lemma~10.40]{MP}. In the Brownian case the exponent $1/3$ can be
calculated by mapping the complement of $\R_+^{2}$ onto a half plane by the
conformal map $z\mapsto z^{2/3}$ and using the conformal invariance of
planar Brownian motion. The continuous case can be transformed to the
discrete one by coupling. For the history of similar estimates and for
further references see Denisov and Wachtel \cite{DW}.

\begin{proof}[Proof of Theorem~\ref{t:2DBM}]
Recall that $S^{(n)}$ are two-dimensional simple random walks so that the rescaled walks
$W_n$ converge uniformly to a Brownian motion $B$. Let $\tau$ be
the hitting time of the positive quadrant by $S$, that is,
$\tau=\inf\{k>0 : S(k) \in \Z_+^2\}$. For every $n\in\N^+$, the greedy
increasing subsequence of $S^{(n)}$ has i.i.d.\ increments, distributed as $\tau$.
By Theorem~\ref{t:greedy} we have
\[
\P(\tau>n)\sim cn^{-1/3}
\]
with some $c\in \R^+$. Thus we can apply Theorem~\ref{t:acc} with
$\alpha=1/3$. This yields that,
almost surely, there is an accumulation point $D$ of $\{D_n\}_{n\geq 1}$
in the Hausdorff metric such that $\dim_\iH D\geq 1/3$. As $W_n \to B$ uniformly, $B$ is non-decreasing on $D$.
This completes the proof.
\end{proof}

\begin{remark}
  For a higher dimensional simple random walk $S\colon \N\to\Z^d$ define the hitting time
  \[
  \tau =\inf\{k>0 : S(k) \in \Z_+^d\}.
  \]
  Then $\P(\tau>n)\sim cn^{-\alpha}$ for some $c,\alpha \in (0,\infty)$, see \cite[Theorem~1]{DW}. Our argument proves an analogue of Theorem~\ref{t:2DBM} in higher dimensions with this $\alpha$ instead of $1/3$.
\end{remark}

\section{Restrictions of generic $\alpha$-H\"older continuous functions} \label{s:generic}

The goal of this section is to prove Theorem~\ref{t:gen}.
First we need some preparation. The following lemma is probably well known.
However, we could not find an explicit reference for its second claim, so we outline the proof.

\begin{lemma} \label{l:ext} Let $0<\alpha\leq 1$ and $c>0$. Assume that $A\subset \R$ and $f\colon A\to \R$ is
 a function such that for all $x,y\in A$ we have
\begin{equation} \label{eq:cxy} |f(x)-f(y)|\leq c|x-y|^{\alpha}.
\end{equation}
Then $f$ extends to $F \colon \R\to \R$ satisfying the above inequality for all $x,y\in \R$.
If $A$ is closed then $F$ can be chosen to be linear on the components of $\R\setminus A$.
\end{lemma}

\begin{proof} As $f$ admits a unique continuous extension to the closure of $A$ which clearly satisfies \eqref{eq:cxy},
we may assume that $A$ is closed. Let $I$ be any component of $\R \setminus A$, it is enough to prove that $f$ extends to
$A\cup I$ such that \eqref{eq:cxy} holds.
If $I=(-\infty,a)$ or $I=(a,\infty)$ for some $a\in A$ then $F|_{I}\equiv f(a)$ works.
Now let $I=(a,b)$ for some $a,b\in A$ and let $F$ be the linear extension of $f$ to $I$.
The concavity of the function $x\mapsto x^{\alpha}$ implies that $|F(x)-F(y)|\leq c|x-y|^{\alpha}$ for all $x,y\in A\cup I$,
the straightforward calculation is left to the reader.
\end{proof}

Let $||\cdot||$ denote the maximum norm of $C[0,1]$.

\begin{lemma} \label{l:dense} Let $f\in C_1^{\alpha}[0,1]$ and $\varepsilon>0$ be arbitrary. Then there is a piecewise
linear function $g$ with nonzero slopes and $c<1$ such that $||g-f||\leq \varepsilon$ and for all $x,y\in [0,1]$ we have
\[|g(x)-g(y)|\leq c|x-y|^{\alpha}.\]
\end{lemma}

\begin{proof} Let $0=x_0<x_1<\dots <x_\ell=1$ such that the oscillation of $f$ on $[x_{i-1},x_i]$ is at most $\varepsilon/3$ for all
$i\in \{1,\dots,\ell\}$. Let $g_0$ be the piecewise linear function passing through the points $(x_i,f(x_i))$, then
clearly $||g_0-f||\leq \varepsilon/3$. Applying Lemma~\ref{l:ext} $\ell$-times we obtain that
$g_0\in C_1^{\alpha}[0,1]$. We can choose $c_0<1$ such that
$g_1=c_0g_0$ satisfies $||g_1-g_0||\leq \varepsilon/3$. Hence for all $x,y\in [0,1]$ we have
\[ |g_1(x)-g_1(y)|\leq c_0|x-y|^{\alpha}.\]
Let $c\in (c_0,1)$, then it is easy to see that
every horizontal line segment of the graph of $g_1$ (if there are any) can be replaced by two line segments of nonzero slopes such that the
resulting function $g$ satisfies $||g-g_1||\leq \varepsilon/3$ and for all $x,y\in [0,1]$ we have
\[ |g(x)-g(y)|\leq c|x-y|^{\alpha}.\]
Clearly $||g-f||\leq \varepsilon$, and the proof is complete.
\end{proof}

Now we are ready to prove the first part of Theorem~\ref{t:gen}.
The concept of the proof is similar to that of \cite[Theorem~1.4]{E},
but the technical details are much more difficult and in order to create appropriate
H\"older continuous functions some new ideas are needed as well.

\begin{proof}[Proof of Theorem~\ref{t:gen}~\eqref{eq:gen1}]
Let $\beta\in (\alpha,1)$ be arbitrarily fixed, and define
\[ \iF_{\beta}=\{f\in C_1^{\alpha}[0,1]: \dim_\iH \{f=g\}\leq 1-\beta \textrm{ for all } g\in C_1^{\beta}[0,1]\},\]
where we use the notation $\{f=g\}=\{x\in [0,1]: f(x)=g(x)\}$.
First we show that it is enough to prove that $\iF_{\beta}$ is co-meager in $C_1^{\alpha}[0,1]$.
Indeed, since co-meager sets are closed under countable intersection, this implies that
for a countable dense set $\Gamma \subset (\alpha,1)$ the set
$\iF:=\bigcap_{\gamma \in \Gamma} \iF_{\gamma}$ is co-meager in $C_1^{\alpha}[0,1]$. Now assume that $f\in \iF$ and
$A\subset [0,1]$ such that $f|_{A}$ is $\beta$-H\"older for some $\beta>\alpha$, we need to prove that $\dim_\iH A\leq 1-\beta$.
Choose a sequence $\gamma_n\in \Gamma$ such that $\gamma_n \nearrow \beta$ and fix an $n\in \N^+$. As $f|_{A}$ is $\beta$-H\"older,
there is an $\varepsilon>0$ such that for all $E\subset A$ with $\diam E\leq \varepsilon$ the function $f|_{E}$ is
$\gamma_n$-H\"older with H\"older constant $1$. Let $A=\bigcup_{i=1}^{k}A_i$ such that $\diam A_i\leq \varepsilon$ for all $i\in \{1,\dots,k\}$.
Then $f|_{A_i}$ are $\gamma_n$-H\"older with H\"older constant $1$, so by Lemma~\ref{l:ext} there are functions
$g_i\in C_1^{\gamma_n}[0,1]$ such that $A_i\subset \{f=g_i\}$ for all $i$.
Therefore $f\in \iF_{\gamma_n}$ implies that $\dim_\iH A_i\leq  \dim_\iH \{f=g_i\}\leq 1-\gamma_n$ for all $i$,
thus the countable stability of Hausdorff dimension yields $\dim_\iH A\leq 1-\gamma_n$. This holds for all
$n\in \N^+$, so $\dim_\iH A\leq 1-\beta$.

Now let $\beta\in (\alpha,1)$ be fixed, and for all $N,M\in \N^+$ define
\begin{equation*}
\iF(N,M)=\left\{f\in C_1^{\alpha}[0,1]: \iH_{\infty}^{1-\beta+1/N}(\{f=g\})\leq 1/M \textrm{ for all }
g\in C_1^{\beta}[0,1] \right\}.
\end{equation*}
Clearly $\iF_{\beta}=\bigcap_{N=1}^{\infty} \bigcap_{M=1}^{\infty} \iF(N,M)$, thus it is enough to show that each $\iF(N,M)$
contains a dense open subset of $C_1^{\alpha}[0,1]$. Assume that $M,N\in \N^+$, $r_0>0$, $c_0<1$,
and a piecewise linear function $f_0\in C_1^{\alpha}[0,1]$ with nonzero slopes are given such that for all $x,y\in [0,1]$ we have
\[ |f_0(x)-f_0(y)|\leq c_0|x-y|^{\alpha}.\]
By Lemma~\ref{l:dense} it is enough to find a function $f\in C_1^{\alpha}[0,1]$ and $r>0$ such that
\begin{equation} \label{eq:Bfr} B(f,r)\subset B(f_0,r_0)\cap \iF(N,M),
\end{equation}
where $B(f,r)$ denotes the closed ball in $C_1^{\alpha}[0,1]$ centered at $f$ with radius $r$.
Now we define $f$. We can fix integers $k_0,m_0\geq 2$ such that $m_0$ is odd and
\[ \max\left\{\alpha, \beta-\frac 1N\right\}<\frac{\log (k_0/2)}{\log (k_0m_0)}<\frac{\log k_0}{\log (k_0m_0)}<\beta.\]
Let $\gamma=\log k_0/\log (k_0m_0)$ and let $f_1=f_{k_0,m_0}\in C^{\gamma}[0,1]$ be the
self-affine function defined in Section~\ref{s:KK}. We will approximate $f_0$ by re-scaled copies of $f_1$.
As $f_1$ is $\gamma$-H\"older continuous, there exists $c_{1}\in \R^+$ such that for all $x,y\in [0,1]$ we have
\begin{equation} \label{eq:ckm} |f_{1}(x)-f_{1}(y)|\leq c_{1}|x-y|^{\gamma}.
\end{equation}
Assume that $0=x_1<\dots<x_{\ell+1}=1$ such that $f_0$ is linear on each interval
$[x_{i},x_{i+1}]$ with nonzero slopes. Let $y_{i}=f_0(x_i)$ for all $1\leq i\leq \ell+1$, then $y_{i+1}-y_i\neq 0$ for all
$i\leq \ell$. Let us define $\theta,\xi>0$ and $n_0\in \N^+$ such that
\begin{align*} \theta&=\min_{1\leq i\leq \ell} (x_{i+1}-x_i), \\
\xi&=\max_{1\leq i\leq \ell} |y_{i+1}-y_i|, \\
n_0&\geq \max\left\{\frac{2\xi}{r_0},\left(\frac{2\xi}{(1-c_0)\theta^{\alpha}}\right)^{1/(1-\alpha)},
\left(\frac{2\xi c_{1}}{\theta^\gamma}\right)^{1/(1-\gamma)} \right\}.
\end{align*}
For all $i\in \{1,\dots,\ell\}$ and $j\in \{0,\dots,n_0\}$ let
\[ x_{i,j}=x_{i}+\frac{j}{n_0}(x_{i+1}-x_{i}) \quad \textrm{and} \quad y_{i,j}=f_0(x_{i,j}).\]
Now we are ready to define $f$. If for some $i\in \{1,\dots,\ell\}$, $j\in \{0,\dots,n_0-1\}$, and
$a\in [0,1)$ we have
\begin{align} \label{eq:xdef} x&=x_{i,j}+\frac {a}{n_0}(x_{i+1}-x_i), \textrm{ then let} \\
f(x)&=y_{i,j}+\frac{f_{1}(a)}{n_0}(y_{i+1}-y_i). \notag
\end{align}
Note that the linearity of $f_0$ implies that if $x$ satisfies \eqref{eq:xdef}, then
\[f_0(x)=y_{i,j}+\frac{a}{n_0}(y_{i+1}-y_i).\]

Now we prove that
\begin{equation} \label{eq:fCB} f\in C_1^{\alpha}[0,1] \quad \textrm{and} \quad f\in B(f_0,r_0/2).
\end{equation}
As the range of $f_0$ is $[0,1]$, the definition of $f$ and $n_0$ imply that for all $x\in [0,1]$ we have
\begin{equation} \label{eq:fxgx} |f(x)-f_0(x)|\leq \frac {\xi}{n_0}\leq \frac{r_0}{2},
\end{equation}
thus it is enough to prove for \eqref{eq:fCB} that $f\in C_1^{\alpha}[0,1]$.
Assume that $x,y\in [0,1]$ and $x<y$, we need to prove that $|f(x)-f(y)|\leq |x-y|^{\alpha}$. We consider three cases.

\noindent \emph{First case}: Suppose that $y-x\geq \theta/n_0$.
Then \eqref{eq:fxgx}, $f_0\in C_1^{\alpha}[0,1]$ and the definition of $n_0$ imply that
\[ |f(x)-f(y)|\leq |f_0(x)-f_0(y)|+\frac{2\xi}{n_0}\leq c_0|x-y|^{\alpha}+\frac{2\xi}{n_0} \leq |x-y|^{\alpha}.\]
\emph{Second case}: Assume that $x,y$ are \emph{adjacent points}, that is, $x,y\in [x_{i,j},x_{i,j+1}]$ for some
$1\leq i\leq \ell$ and $0\leq j\leq n_0-1$. Then the definitions of $f$ and $\xi$, inequality \eqref{eq:ckm}, the definition of $n_0$, and $\alpha<\gamma$ yield that
\begin{align*} |f(x)-f(y)|&=\frac{|y_{i+1}-y_i|}{n_0}  \left|f_{1}\left(n_0\frac{x-x_{i,j}}{x_{i+1}-x_i}\right)-
f_{1}\left(n_0\frac{y-x_{i,j}}{x_{i+1}-x_i}\right)\right| \\
&\leq \frac{\xi c_1}{n_0} \left(\frac{n_0|x-y|}{\theta}\right)^{\gamma}\leq |x-y|^{\gamma}\leq |x-y|^{\alpha}.
\end{align*}
\emph{Third case}: Suppose that there exists $z\in (x,y)$ such that $x,z$ and $z,y$ are adjacent points.
The triangle inequality, the above inequality, the definition of $n_0$, and $\alpha<\gamma$ imply that
\begin{align*} |f(x)-f(y)|&\leq |f(x)-f(z)|+|f(z)-f(y)| \\
&\leq \frac{\xi c_1}{n_0}  \left(\left(\frac{n_0|x-z|}{\theta}\right)^{\gamma}+ \left(\frac{n_0|z-y|}{\theta}\right)^{\gamma}\right)\\
&\leq \frac{2\xi c_1}{n_0}  \left(\frac{n_0|x-y|}{\theta}\right)^{\gamma} \leq |x-y|^{\gamma}\leq |x-y|^{\alpha}.
\end{align*}
By the definition of $\theta$ for all $x,y$ at least one of the above three cases holds,
which concludes the proof of \eqref{eq:fCB}.

Finally, we prove that \eqref{eq:Bfr} holds for some $r>0$. Let us define $\delta>0$ as
\[ \delta=\min_{1\leq i\leq \ell} |y_{i+1}-y_i|.\]
Since $\gamma<\beta$, we have $k_0<(k_0m_0)^{\beta}$. Thus by $\beta<1$ we can fix an $n_1\in \N^+$
such that for all $i\in \N^+$ and $n\geq n_1$ we have
\begin{equation} \label{eq:n1}
\left(\frac{i+2}{(k_0m_0)^{n}}\right)^{\beta}< \frac{i\delta}{3n_0k_0^{n}}.
\end{equation}
Since $\log (k_0/2)/\log (k_0m_0)>\beta-1/N$,
we can define
\[ \sigma=\frac{2m_0}{(k_0m_0)^{1-\beta+1/N}}<1.\]
Let us fix an integer $n_2>n_1$ such that
\begin{equation} \label{eq:n2} \sigma^{n_2}<\frac{1}{k_0^{n_1} M\ell n_0}.
\end{equation}
Let us define $r>0$ as
\[ r=\min\left\{\frac{r_0}{2},\frac{\delta}{3n_0k_0^{n_2}}\right\}.\]
Then clearly $B(f,r)\subset B(f_0,r_0)$. Let us fix $g\in B(f,r)$, it is enough to show that $g\in \iF(N,M)$. Fix
$i\in \{1,\dots,\ell\}$ and $j\in \{0,\dots,n_0-1\}$ arbitrarily, and let $I_0=[x_{i,j},x_{i,j+1}]$.
Let $h\in C_1^{\beta}[0,1]$, by the subadditivity of $\iH_{\infty}^{s}$ it is enough to show that
\begin{equation} \label{eq:Mn0}
\iH_{\infty}^{1-\beta+1/N}(\{g=h\}\cap I_0)\leq \frac{1}{M\ell n_0}.
\end{equation}
Assume that $n\in \N$ and $I\subset I_0$ is a closed interval.
We divide $I$ into $(k_0m_0)^{n}$ non-overlapping closed intervals of equal length,
the resulting intervals are called the \emph{elementary intervals of $I$ of level $n$}.
Assume that $n_1\leq n\leq n_2$ and let $I_1$ be an elementary interval of $I_0$ of level $n-1$. Now we show that
$\{g=h\}$ intersects at most $2m_0$ many first level elementary intervals of $I_1$.
Let us decompose $I_1$ into $m_0$ non-overlapping intervals of equal length, let
$I_2$ be one of them. Let $J_1,J_2\subset I_2$ be two nonconsecutive first level elementary intervals of $I_1$,
it is enough to show that $\{g=h\}$ cannot intersect both $J_1$ and $J_2$.
Assume to the contrary that there are $z_1\in J_1$ and $z_2\in J_2$ such that
$g(z_1)=h(z_1)$ and $g(z_2)=h(z_2)$.
If there are $i\in \{1,\dots, k_0-2\}$ first level elementary intervals of $I_1$ between
$J_1$ and $J_2$, then
\[ |z_1-z_2|\leq (i+2)\diam J_1\leq (i+2)(k_0m_0)^{-n}.\]
Therefore $h\in C_1^{\beta}[0,1]$ yields that
\[ |h(z_1)-h(z_2)|\leq \left(\frac{i+2}{(k_0m_0)^{n}}\right)^{\beta}.\]
On the other hand side, the definition of $g$, $f$ and $r$ imply that
\[ |g(z_1)-g(z_2)|\geq |f(z_1)-f(z_2)|-2r\geq  i\frac{\delta}{n_0k_0^{n}}-2r\geq \frac{i\delta}{3n_0k_0^{n}}.\]
The above inequalities and \eqref{eq:n1} yield that $|h(z_1)-h(z_2)|<|g(z_1)-g(z_2)|$, which is a contradiction.

Therefore
$\{g=h\}\cap I_0$ intersects at most $(k_0m_0)^{n_1}(2m_0)^{n_2-n_1}$ many elementary intervals of $I_0$ of level $n_2$. Since
the length of these intervals is less than $(k_0m_0)^{-n_2}$, the definition of $\sigma$ and
inequality \eqref{eq:n2} yield that
\begin{align*} \iH_{\infty}^{1-\beta+1/N}(\{g=h\}\cap I_0)&\leq (k_0m_0)^{n_1}(2m_0)^{n_2-n_1}
 (k_0m_0)^{-n_2(1-\beta+1/N)} \\
&\leq k_0^{n_1}\sigma^{n_2}\leq \frac{1}{M\ell n_0}.
\end{align*}
Hence \eqref{eq:Mn0} holds, and the proof is complete.
\end{proof}

In order to prove the second part of Theorem~\ref{t:gen} we need a bridge between
the notions of $\alpha$-H\"older continuity and $\beta$-variation.
The following lemma is \cite[Lemma~4.1]{BP}, see also \cite[Lemma~4.1]{ABPR}.

\begin{lemma} \label{l:BV}
Let $\beta,\gamma>0$ and let $A\subset [0,1]$. If the function $f\colon A \to \R$ has finite $\beta$-variation,
then there are sets $A_n\subset A$ such that
\begin{enumerate}
\item $f|_{A_n}$ is $\gamma$-H\"older continuous for all $n\in \N^+$,
\item $\dim_\iH \left(A\setminus \bigcup_{n=1}^{\infty} A_n\right)\leq \gamma \beta$.
\end{enumerate}
\end{lemma}

\begin{proof}[Proof of Theorem~\ref{t:gen}~\eqref{eq:gen2}]
Clearly it is enough to prove the theorem for a countable dense set of parameters $\beta$. Since
co-meager sets are closed under countable intersection, it is enough to show the statement for an arbitrary fixed $\beta>0$.
For all $f\in C_1^{\alpha}[0,1]$ let $A_f\subset [0,1]$ be given such that $V^{\beta}(f|_{A_f})<\infty$.
Fix an arbitrary $\delta>\max\left\{\alpha\beta, \beta/(\beta+1)\right\}$ and let $\gamma=\delta/\beta>\max\left\{\alpha, 1/(\beta+1)\right\}$.
It is enough to prove that $\dim_\iH A_f\leq \delta$ for a generic $f\in C_1^{\alpha}[0,1]$.
Applying Lemma~\ref{l:BV} we obtain that for all $f\in C_1^{\alpha}[0,1]$
there are sets $A_{f,n}\subset A_f$ such that
\begin{enumerate}
\item $f|_{A_{f,n}}$ is $\gamma$-H\"older continuous for all $n\in \N^+$,
\item \label{eq:alpha} $\dim_\iH \left(A_f\setminus \bigcup_{n=1}^{\infty} A_{f,n}\right)\leq \gamma \beta=\delta$.
\end{enumerate}
As $\gamma>\alpha$ and $f|_{A_{f,n}}$ are $\gamma$-H\"older continuous, Theorem~\ref{t:gen}~\eqref{eq:gen1} and the definition of $\gamma$ imply
that for a generic $f\in C_1^{\alpha}[0,1]$ for all $n\in \N^+$ we have
\begin{equation} \label{eq:delta} \dim_\iH A_{f,n}\leq 1-\gamma < 1-\frac{1}{\beta+1}=\frac{\beta}{\beta+1}<\delta.
\end{equation}
Inequalities \eqref{eq:alpha}, \eqref{eq:delta}, and the countable stability of Hausdorff dimension
yield that $\dim_\iH A_f\leq \delta$ for a generic $f\in C_1^{\alpha}[0,1]$. The proof is complete.
\end{proof}

\end{document}